\documentclass[]{siamart220329}
\usepackage{amsmath}
\usepackage{amssymb}
\usepackage{algorithm}
\usepackage{algorithmic}
\usepackage{subfig}
\usepackage{graphicx}
\usepackage{mdframed}
\usepackage{mathtools}
\usepackage{float}
\usepackage{multirow}
\usepackage{matlab-prettifier}
\usepackage{algorithm}
\usepackage{algorithmic}

\newsiamremark{remark}{Remark}
\newsiamremark{hypothesis}{Hypothesis}
\crefname{hypothesis}{Hypothesis}{Hypotheses}
\newsiamthm{claim}{Claim}

\usepackage{graphicx} 
\usepackage{amsfonts,amssymb,amsmath}
\usepackage{cite} 
\usepackage{color}
\usepackage[font=small,labelfont=bf]{caption}

\usepackage{algorithm}
\usepackage{algorithmic}
\usepackage{subcaption}

\DeclareMathAlphabet{\mathbf}{OT1}{cmr}{bx}{it}

\newcommand{\vb}{\mathbf b}

\newcommand{\ve}{\mathbf e}
\newcommand{\vf}{\mathbf f}

\newcommand{\vr}{\mathbf r}

\newcommand{\vt}{\mathbf t}
\newcommand{\vu}{\mathbf u}
\newcommand{\vv}{\mathbf v}
\newcommand{\vw}{\mathbf w}
\newcommand{\vx}{\mathbf x}
\newcommand{\vy}{\mathbf y}
\newcommand{\vz}{\mathbf z}

\newcommand{\vU}{\mathbf U}
\newcommand{\vV}{\mathbf V}
\newcommand{\vW}{\mathbf W}

\newcommand{\veta}{\boldsymbol \eta}
\newcommand{\CK}{\mathcal{K}}
\newcommand{\rng}{\mathrm{Range}}
\newcommand{\rnga}[1]{\mathrm{Range}\left( #1 \right)}
\newcommand{\ang}[2]{\theta\left( #1,\,#2 \right)}
\newcommand{\argmin}[1]{ \underset{#1}{\mathrm{argmin}}}
\newcommand{\norm}[1]{\left\|#1\right\|}

\DeclareMathOperator{\spn}{span}

\newcommand{\rev}[1]{#1}

\title{GMRES with randomized sketching and deflated~restarting}
\author{Liam Burke\thanks{School of Mathematics, Trinity College Dublin, College Green, Dublin 2, Ireland, \email{burkel8@tcd.ie}. L.\,B. acknowledges funding from an Irish Research Council Government of Ireland Postgraduate Scholarship.}
\and Stefan G\"{u}ttel\thanks{Department of Mathematics, The University of Manchester, M13 9PL Manchester, United Kingdom, \email{stefan.guettel@manchester.ac.uk}}.  S.\,G.'s work is partly funded by a Royal Society Industry Fellowship IF/R1/231032.
\and
Kirk M.~Soodhalter\thanks{School of Mathematics, Trinity College Dublin, College Green, Dublin 2, Ireland, \email{ksoodha@maths.tcd.ie}.
}}

\usepackage{amsopn}

\ifpdf
\hypersetup{
  pdftitle={Sketch-and-recycle GMRES},
  pdfauthor={authors}
}
\fi

\begin{document}

\maketitle

\begin{abstract}
\rev{
We present a new Krylov subspace recycling method for solving a linear system of equations, or a sequence of slowly changing linear systems.  Our approach is to reduce the computational overhead of recycling techniques while still benefiting from the acceleration afforded by such techniques. As such, this method augments an unprojected Krylov subspace. Furthermore, it combines randomized sketching and deflated restarting in a way that avoids orthogononalizing a full Krylov basis.  We call this new method GMRES-SDR (sketched deflated restarting).  With this new method, we provide new theory, which initially characterizes unaugmented sketched GMRES as a projection method for which the projectors involve the sketching operator.  We demonstrate that sketched GMRES and its sibling method sketched FOM are an MR/OR pairing, just like GMRES and FOM.  We furthermore obtain residual convergence estimates.  Building on this, we characterize GMRES-SDR also in terms of sketching-based projectors.  Compression of the augmented Krylov subspace for recycling is performed using a sketched version of harmonic Ritz vectors.  We present results of numerical experiments demonstrating the effectiveness of GMRES-SDR over competitor methods such as GMRES-DR and GCRO-DR.
}
\end{abstract}


\overfullrule=0pt 

\section{Introduction}
This paper is concerned with the development of a Krylov subspace recycling algorithm for
the efficient solution of a linear system of equations
\[
    A \vx = \vb
\]
with a nonsymmetric matrix $A \in \mathbb{C}^{N \times N}$ and a right-hand side vector $\vb \in \mathbb{C}^{N}$. Our algorithm can also naturally be applied to solve sequences of nearby linear systems
\begin{equation}\label{eq:sequence_of_linear_systems}
    A^{(i)} \vx^{(i)} = \vb^{(i)}, \qquad i=1,2,\ldots,
\end{equation}
using accumulated Krylov information computed to improve convergence as the sequence progresses. Such problems arise in a variety of scientific computing applications (see, e.g, \cite{parks2006recycling}), and 
Krylov subspace recycling has been demonstrated to be a successful technique in reducing the computational resources and runtime required to solve the full sequence of problems~\cite{eveolving_structures, gaul2014recycling}.
Recycling algorithms belong to the class of augmented Krylov subspace methods, where the augmentation subspace for each problem is constructed or \textit{recycled} from the Krylov subspace used to solve a previous problem in the sequence. Although recycling introduces additional computational overhead per iteration when compared to non-augmented methods, the presence of the recycling subspace can often allow each problem to converge faster and in much less iterations, resulting in a reduction in overall computational cost and runtime for the full problem sequence.

One of the dominating computational kernels of common restarting or recycling methods, such as GMRES-DR~\cite{Morgan2002} or GCRO-DR~\cite{parks2006recycling}, is the orthogonalization of a full Krylov basis in each cycle. In this paper, we build on  existing work in \cite{NakatsukasaTropp21,BalabanovGrigori22,GS23,BG23} to develop a new GMRES-based recycling method, named GMRES-SDR, which reduces the expensive orthogonalization costs using the well established technique of randomized sketching~\cite{woodruff2014sketching}.  

\rev{
In order to fully describe the theoretical underpinnings of a sketched recycling method like GMRES-SDR, we expand the existing theory of sketched GMRES.  We show that it is indeed a projection method wherein the projector involves the sketching operator.  This allows us to build on the projection framework for recycling methods presented in \cite{Soodhalter:deSturler:Kilmer.2020.framework} to describe and understand the GMRES-SDR.  We are able to show that the sketched recycling methods fit into this same framework, again with the associated projectors involving the sketching operator.  We also expand the theory of sketched GMRES convergence to build upon when analysing GMRES-SDR.  Namely, we show that the convergence behavior of sketched GMRES is closely linked to subspace angles and Givens rotations, just as was shown in \cite{EiermannErnst2001} for GMRES.  Indeed, sketched GMRES and sketched FOM are shown to share the same MR/OR relationship and GMRES and FOM, as described in \cite{Brown1991}.
}

The structure of the paper is as follows. 
In \Cref{sec:background} we briefly review some of the most commonly used recycling methods and introduce our notation. \Cref{sec:sdr} introduces the new GMRES-SDR method, including the implementation details.  As analysis of augmented and recycled Krylov subspace methods is often built on top of the analysis of the method being augmented, in \Cref{section:sketched-gmres-analysis} we develop some additional theory about sketched GMRES as a projection-type scheme and its behavior.  We build on this in \Cref{section:sketched-aug-gmres-analysis} by showing that sketched augmented GMRES exhibits many of the features of an augmented projection method as laid out in \cite{Soodhalter:deSturler:Kilmer.2020.framework}, and we extend the analysis of sketched GMRES developed in \Cref{section:sketched-gmres-analysis} to get some computable convergence bounds for sketched augmented GMRES.  In \Cref{sec:numerical-experiments}, we demonstrate that the sketched augmented GMRES method we propose in this manuscript is effective and reduces time-to-solution over other augmented and recycled methods for some large-scale problems arising in the computational sciences. 

\rev{
In this work, we frequently describe the workings of various (augmented) projection methods in terms of the actions of projectors.  It is important in this context to differentiate between the quantities/subspaces associated to the solution and error and those associated to the right-hand side and residual.  Any projection method we discuss can be described either in terms of projectors acting on the error or equivalently in terms of a related projector acting on the residual.  These related projectors always come in pairs $(\Pi,\Phi)$. In this manuscript, any projector acting on an error quantity is denoted by $\Pi$ with a subscript, where the subscript indicates the range of the projector.  The related projector acting on a residual quantity is denoted by $\Phi$ with a subscript, and it shares the same subspace-related subscript with its pair-mate.  For example, classic GMRES (cf. \cref{section:GMRES-classic}) over a Krylov subspace $\CK_m$ can be shown to generate an approximation $\vx_m = \vx_0 + \Pi_{\CK_m}\veta_0$ (with $\veta_0$ being the initial error) such that the residual $\vr_m = \vb-A\vx_m = \left( I - \Phi_{\CK_m} \right)\vr_0$, where $\Pi_{\CK_m}$ is the $A^\ast A$-orthogonal projector onto $\CK_m$ and $\Phi_{\CK_m}$ is the orthogonal projector onto $A\CK_m$.
In the case that projectors associated to sketched version of a method are discussed, these projectors are denoted with hats (e.g., $\widehat{\Phi}_{\CK_m}$ for sketched GMRES).  
}

\section{Relevant background}\label{sec:background}
This paper concerns augmented and sketched \rev{variants} of GMRES, and so we need to introduce some relevant theory and notation concerning GMRES and augmented/recycled \rev{variants} thereof. This discussion is abbreviated, focusing on the aspects we need to develop and analyze the methods we propose in subsequent sections.

\subsection{The classic GMRES method}\label{section:GMRES-classic}
We describe one cycle of classic GMRES~\cite{SaadSchultz1986} for solving a single linear system $A \vx = \vb$ with initial guess $\vx_0$ and associated residual $\vr_0 := \vb - A\vx_0$.  GMRES generates an orthonormal basis for the Krylov space $\CK_m( A,\vr_0) := \spn\{\vr_0,A\vr_0, \ldots,A^{m-1}\vr_0\}$, which are stored as the columns of a matrix $V_m\in\mathbb{C}^{N\times m}$.  This basis is generated by the modified Gram--Schmidt process, satisfying the Arnoldi relation
\begin{align}
     A V_m
    =
    V_{m+1}\underline{H_m},
    \label{eq:Arnoldi}
\end{align}
    where $\underline{H_m}\in\mathbb{C}^{(m+1)\times m}$ is upper Hessenberg.  GMRES computes the minimizer
    \begin{align*}
        \vx_m 
        = 
        \argmin{\vt\in\CK_m( A,\vr_0)}
        \norm
        {
            \vr_0 -  A\vt
        },
    \end{align*}
    which is equivalent to setting $\vx_m = V_m\vy_m $ with 
    \begin{align*}
        \vy_m 
        = 
        \argmin{\vy\in\mathbb{C}^m}
        \norm
        {
            \underline{H_m}\vy - \norm{\vr_0}\ve_1
        }
        .
    \end{align*}
    The GMRES least squares approximation for $A\vx = \vb$ using any basis  $V_m\in\mathbb{C}^{N\times m}$ can be written 
    as $\vx_m = \vx_0 +  V_m \vy_m$, where $\vy_m=\left(  V_m^\ast A^\ast A V_m \right)^{-1} V_m^\ast A^\ast\vr_0$.  The residual can thus be written  as 
    \begin{align*}
        \vr_m = (I -  A V_m\left(  V_m^\ast A^\ast A V_m \right)^{-1} V_m^\ast A^\ast)\vr_0,
    \end{align*}
    which is an orthogonal projection of $\vr_0$ onto $\rng( A V_m)^\perp$, leading to the well-known fact that $\rng( A V_m)$ is the GMRES residual constraint space.  Let $\Phi_{\CK_m} =  A V_m\left(  V_m^\ast A^\ast A V_m \right)^{-1} V_m^\ast A^\ast$ be the orthogonal projector onto $\rng( A V_m)$,  which we call a \emph{residual projector}.  There is a sibling \emph{error projector,} \linebreak $\Pi_{\CK_m} =  V_m\left(  V_m^\ast A^\ast A V_m \right)^{-1} V_m^\ast A^\ast A$, which is the $ A^\ast A$-orthogonal projector onto the space $\CK_m( A,\vr_0)$ so that $\vx_m = \vx_0 + \Pi_{\CK_m}\veta_0$, where $\veta_0 = \vx-\vx_0$.
    
    \subsection{Recycling and augmentation of minimum residual methods}\label{section:augmentation-framework}
    In \cite{Soodhalter:deSturler:Kilmer.2020.framework} it was established that augmented subspace methods (such as augmented Krylov methods and recycling methods) can all be characterized in a general \emph{projection--correction} framework.  In the case of a minimum residual method, choosing $\vf_m\in\rng\left(\vV_m\right)$ for $\vV_m = [ U, V_m]$ such that  
    \begin{align*}
        \vf_m 
        = 
        \argmin{\vf\in\rng\left(\vV_m\right)}
        \norm{ 
            \vb -  A\left( \vx_0 + \vf \right) 
        }
    \end{align*}
    is equivalent to selecting 
    \begin{align*}
        \vt_m
        =
        \argmin{\vt\in\rng\left( V_m \right)}
        \norm{ 
            \left( I - \Phi_U \right)\left( \vb -  A\vt \right) 
        },
    \end{align*}
    where $\Phi_U$ is the orthogonal residual projector onto $\rng( A U)$, and constructing $\vx_m = \vx_0 + \Pi_U\veta_0 + (I-\Pi_U)\vt_m$, where $\Pi_U$ is the sibling $ A^\ast A$-orthogonal error projector onto $\rnga{U}$.  In other words, the augmented minimum residual method approach is equivalent to applying a minimum residual method to the projected subproblem
    \begin{align}
        \left( I - \Phi_U \right)A\vt
        =
        \left( I - \Phi_U \right)\vb
        \label{eq:proj-subprobem}
    \end{align}
    over the approximation space $\rnga{V_m}$ and then applying some projections back into $\rnga{U}$ to construct $\vx_m$.

    \subsection{GMRES-DR and GCRO-DR are sometimes equivalent}\label{section:gmresdr-gcrodr-equiv}
    It was observed in \cite{Parks:2005-Thesis} that GCRO-DR \cite{parks2006recycling} and GMRES-DR \cite{Morgan2002} are mathematically equivalent in limited circumstances.
    
    The GCRO-DR method \cite{parks2006recycling} takes $\rnga{V_m} = \CK_m\left(\left( I - \Phi_U \right)A,\left( I - \Phi_U \right)\vr_0\right)$ with $V_m$ being built by the Arnoldi process.  Thus it satisfies the Arnoldi relation for the projected operator,
    \begin{align*}
        \left( I - \Phi_U \right)A V_m 
        = 
        V_{m+1}\underline{H_m}
        .
    \end{align*}
    From this, it follows that the minimization is equivalent to applying a GMRES iteration to \eqref{eq:proj-subprobem}.  This method accommodates beginning the iteration with an arbitrary~$U$, and any technique of selecting vectors to recycle can be used.

    The GMRES-DR method \cite{Morgan2002} takes a similar approach but is restricted to reusing subspace information for a single linear system between restarts.  At the end of a cycle, the method computes harmonic Ritz vectors which are then orthogonalized to save for the next cycle. Let $U$ denote the matrix with the harmonic Ritz vectors as columns.  The last GMRES residual $\vr_m$ is orthogonalized to produce $\vv_{k+1} = (I - UU^\ast)\vr_m$.  These orthogonalizations are performed efficiently by taking advantage of the Arnoldi relations of the just-completed GMRES cycle.  It is proven that if harmonic Ritz vectors are used, the space $\mathrm{span}\left\lbrace \vu_1, \vu_2, \ldots, \vu_k, \vv_{k+1} \right\rbrace$ is a Krylov subspace and an Arnoldi relation for it can be cheaply built.  Thereafter, the iteration continues.  Because the augmentation space $\rnga{U}$ can be generated so that the augmented subspace remains a Krylov subspace, GMRES-DR has a leaner implementation than GCRO-DR.  In particular, with a proper Arnoldi relation, only $U$ needs to be stored.  There is no need to store $AU$.

    Consider running a cycle of GMRES, stopping to restart, and computing some harmonic Ritz vectors to carry over to the next cycle.  GMRES-DR proceeds in the next cycle by constructing an orthonormal basis for $\vV_m$ which is established to be a Krylov subspace (since $U$ consists of harmonic Ritz vectors from the previous cycle).  We observe that $\vv_j$ for $j = k+2, k+3, \cdots, m$ is generated by applying $A$ to $\vv_{j-1}$, orthogonalizing against the columns of $U$, and orthogonalizing against $\vv_{i}$ for $i=k+1, k+2, \ldots, j-1$, i.e., perform the Arnoldi process. This means that $\left\lbrace \vv_{k+1}, \vv_{k+2}, \ldots, \vv_{j-1}, \vv_j\right\rbrace$ forms a basis for $\CK_{j-k-1}\left( (I - UU^\ast)A, (I - UU^\ast)\vr_m \right)$, the space which would be built by GCRO-DR in this setting.  One observes this by noting that $U$ being built from harmonic Ritz vectors means they satisfy an Arnoldi relation
$
        AU
        =
        [
            U , \vv_{k+1}
        ]
        \underline{\widetilde{H}_m},
$  
    where Morgan \cite{Morgan2002} derives a cheap procedure for obtaining the upper Hesssenberg matrix~$\underline{\widetilde{H}_m}$.

\section{Derivation of GMRES-SDR}\label{sec:sdr}
\rev{In this section, we briefly review the basics of sketched GMRES~\cite[eq.~(1.6)]{NakatsukasaTropp21} before developing GMRES-SDR.

Let us first consider a single linear system $A \vx = \vb$ with initial guess $\vx_0$ and associated residual $\vr_0 := \vb - A\vx_0$. 
We will assume the columns of $\vV_{m}$ span some search space.}
In order to exploit randomized sketching (see, e.g., \cite{martinsson2020randomized,balabanov2019randomized,balabanov2021randomized,BalabanovGrigori21,NakatsukasaTropp21}), we further assume that we have an operator $S\in\mathbb{C}^{s\times N}$ with $m < s\ll N$ which acts as an approximate isometry for the Euclidean norm $\|\,\cdot\,\|$. More precisely, given a positive integer $m$ and some $\varepsilon\in [0,1)$, let $S$ be such that for all vectors~$\vv\in\spn(\vV_m)$,
\begin{equation}
(1-\varepsilon) \| \vv \|^2 \leq \| S \vv\|^2 \leq (1+\varepsilon) \|\vv\|^2.
\label{eq:sketch}
\end{equation}
The mapping $S$ is called an \emph{$\varepsilon$-subspace embedding} for $\spn(\vV_m)$; see, e.g.,~\cite{sarlos2006improved,woodruff2014sketching,martinsson2020randomized}. 
Condition~\eqref{eq:sketch} can equivalently be stated with the Euclidean inner product~\cite[Cor.~4]{sarlos2006improved}: for all $\vu,\vv \in \spn(\vV_m)$,  
\begin{equation}{\label{eq:sketch_innerproduct}}
\langle \vu, \vv \rangle - \varepsilon \| \vu\|\cdot \|\vv\|
            \leq \langle S\vu, S\vv \rangle 
            \leq \langle \vu, \vv \rangle + \varepsilon \| \vu\|\cdot \|\vv\|.\nonumber
\end{equation}
In practice, $S$ is not explicitly available but we can draw it at random to achieve~\eqref{eq:sketch} with high probability.

\rev{ In the case of sketched GMRES, the columns of $\vV_{m}$ form a (not necessarily orthogonal) basis for the Krylov subspace $\mathcal{K}_m(A,\vr_0) = \spn\{\vr_0,A\vr_0, \ldots,A^{m-1}\vr_0\}$. The algorithm computes a correction $\vt_{m} \in \spn (\vV_{m})$ such that the associated sketched residual $S \vr_{m}$ satisfies $ S \vr_{m} \perp S A \vV_{m}$, leading to the updated solution
\begin{equation}\label{eq:sGMRES_problem}
  \vx_{m} = \vx_{0} + \vV_{m} \vy_{m}, \enspace    \vy_{m} = \argmin{\vy \in \spn(\vV_{m})}
        \norm
        {
            S A \vV_{m} \vy - S \vr_{0}
        }.
\end{equation}
The size of the sketched problem \eqref{eq:sGMRES_problem} depends only on $s$ and $m$, and can thus be solved cheaply without a full orthogonalization of the basis $\vV_{m}$. This is the  attractive feature of sketched GMRES which serves as our motivation for incorporating sketching into recycled GMRES.

In this work, we are interested in the case where $\vV_m = [U,V_m]$ where $U$ is an augmentation subspace, and $V_m$ is a (not necessarily orthogonal) Krylov basis of $\mathcal{K}_m(A,\vr_0)$.} By \cite[Remark~3.1]{BG23}, the \rev{sketched} augmented GMRES-type approximant for a matrix function $f(A)\vb$ is 
\begin{equation}
\widetilde{\vx}_m = \vV_m f\left([(S\vW_m)^* S\vV_m]^{-1} (S\vW_m)^* S\vW_m\right) [(S\vW_m)^* S\vV_m ]^{-1} (S\vW_m)^*S\vr_0,
\label{eq:sgmresf}
\end{equation}
with $\vW_m = A\vV_m$. Here we have $f(z) = z^{-1}$, in which case this simplifies to
\[
\widetilde{\vx}_m = \vV_m \left[(SA\vV_m)^* SA\vV_m\right]^{-1}  (SA\vV_m)^*S\vr_0,
\]
or alternatively,
\begin{equation}
\widetilde{\vx}_m = \vV_m \widetilde{\vy}_m, \quad \widetilde{\vy}_m = (SA\vV_m)^\dagger (S\vr_0).
\label{eq:sgmres}
\end{equation}
This means that 
\[
\widetilde{\vx}_m \ \text{minimizes} \ \texttt{sres} := \| S\vr_0 - SA\widetilde \vx\| \ \text{over all} \  \widetilde \vx \in\spn(\vV_m), 
\]
\rev{and we note that this is also the definition of the sGMRES approximant.}
The interesting aspect here is that this approximant does not rely on orthogonality of $\vV_m$, yet its residual $\widetilde{\vr}_m=\vr_0 - A\widetilde{\vx}_m$ is closely related to the residual $\vr_m = \vr_0 - A\vx_m$ of the full GMRES approximant $\vx_m := \vV_m (A\vV_m)^\dagger \vr_0$. This latter approximant
\[
{\vx}_m \ \text{minimizes} \   \| \vr_0 - A\widetilde \vx\| \ \text{over all} \  \widetilde \vx \in\spn(\vV_m). 
\]
Combining this with \eqref{eq:sketch}, we have the following residual relations
\begin{equation}
        \|  \vr_m\| \leq {\|\widetilde \vr_m\|} \leq \frac{1}{\sqrt{1-\varepsilon}}  \|S \widetilde \vr_m \| \leq  \frac{1}{\sqrt{1-\varepsilon}}  \|S \vr_m \|
         \leq \sqrt{\frac{1+\varepsilon}{1-\varepsilon}} \| \vr_m \|.
         \label{eq:reschain}
\end{equation}

\subsection{Implementation}\label{sec:impl} The GMRES-SDR method is now fairly straightforward to implement, as detailed in Algorithm~\ref{alg:gmres_sdr}. There are several  components that differ from the usual GMRES implementation and we discuss them separately in the following subsections.

\begin{algorithm}[t]
\caption{One cycle of GMRES-SDR for $A\vx = \vr_0$ \label{alg:gmres_sdr}}
\begin{algorithmic}[1]
\STATE{\textbf{Input:} ${A}\in\mathbb{C}^{N\times N}$, $\vr_0\in\mathbb{C}^{N}$,  
$U\in\mathbb{C}^{N\times \widehat k}$, $S\in\mathbb{C}^{s\times N}$, $\texttt{tol}$, integers $m$, $t$, $k$}
\STATE{\textbf{Optional input:} $SU \in\mathbb{C}^{s\times \widehat k}$, $SAU \in\mathbb{C}^{s\times \widehat k}$, $\texttt{safety}\geq 1$ (default 1.4)}\\[1mm]
\STATE{\textbf{Output:} Approximant $\widetilde\vx \approx A^{-1} \vr_0$ and recycling subspace $U,SU,SAU$}\\[1mm]
\STATE{If not provided as input, compute sketches $SU$ and $SAU$}
\STATE{Sketch initial residual $S\vr_0$}
\STATE{$V :=[ \vr_0/\| \vr_0\|, O_{N\times m} ]$, $H:=O_{(m+1)\times m}$}
\STATE{$SV :=[ S\vr_0/\| \vr_0\|, O_{N\times m} ]$, $SAV:=O_{s\times m}$ }
\STATE{\% ------ Truncated Arnoldi process with QR-based residual estimation ------}
\STATE{\textbf{for} $j = 1:m$}
\STATE{\qquad $\vw:= A V(\,:\,,j)$}\\
\STATE{\qquad \textbf{for} $i=\max(j-t+1,1)\rev{\ :j}$}
\STATE{\qquad \qquad $H(i,j) := V(\,:\,,i)^* \vw$}
\STATE{\qquad \qquad $\vw  := \vw - V(\,:\,,i) H(i,j)$}
\STATE{\qquad $H(j+1,j) := \|\vw\|$}
\STATE{\qquad $V(\,:\,,j+1) := \vw/H(j+1,j)$}
\STATE{\qquad $SV(\,:\,,j+1) := S\cdot V(\,:\,,j+1)$ \quad \quad\quad\quad \quad\quad\quad\quad \quad\% one sketch per iter}
\STATE{\qquad $SAV(\,:\,,j) := SV(\,:\,,1:j+1) H(1:j+1,j)$}
\STATE{\qquad $SAW := [ SAU , SAV(\,:\,,1:j) ]$}
\STATE{\qquad Compute (or update) economic $[Q,R] = \texttt{qr}( SAW )$} 
\STATE{\qquad Compute least squares solution $\widetilde\vy := R^{-1} Q^* [S\vr_0]$}
\STATE{\qquad $\texttt{sres} := \| S\vr_0 - SAW \cdot \widetilde\vy \|$\quad\quad\quad\quad\quad\quad\quad\quad\quad\quad\quad \% sketched residual} 
\STATE{\qquad \textbf{if} $\texttt{sres} < \texttt{tol}/\texttt{safety}$ \textbf{or} $j=m$}
\STATE{\qquad \qquad $\widetilde\vx := [U,V(\,:\,,1:j)]\cdot \widetilde\vy$}
\STATE{\qquad \qquad $\texttt{res}:= \| \vr_0 - A\widetilde\vx \|$ \quad\quad\quad\quad\quad\quad\quad\quad\quad\quad\quad\quad\ \% true residual} 
\STATE{\qquad \qquad \textbf{if} $\texttt{res} < \texttt{tol}$}
\STATE{\qquad \qquad \qquad \textbf{break} \quad\quad\quad\quad\quad\quad\quad\quad\quad\quad\quad\quad\quad\quad\quad\ \% stop trunc.\ Arnoldi}
\STATE{\qquad \qquad \textbf{else}}
\STATE{\qquad \qquad \qquad $\texttt{safety}:=\texttt{res}/\texttt{sres}$ \quad\quad\quad\quad\quad\quad\quad\quad\quad\,\% increase safety factor}
\STATE{\% --------- SVD-based update of recycling space using harmonic Ritz ---------}
\STATE{Compute truncated economic SVD $\widehat U_\ell\Sigma_\ell \widehat V_\ell^* \approx SAW$}
\STATE{Compute $M_\ell := \widehat U_\ell^* \cdot SW\cdot \widehat V_\ell$}
\STATE{Compute ordered QZ decomposition $[Q,Z] = \texttt{qz} (M_\ell,\Sigma_\ell)$}
\STATE{Compute $U: = [U,V(\,:\,,1:m)] \widehat V_\ell Z(\,:\,,1:k)$}
\STATE{Compute $SU: = [SU,SV(\,:\,,1:m)] \widehat V_\ell Z(\,:\,,1:k)$}
\STATE{Compute $SAU: = [SAU,SAV(\,:\,,1:m)] \widehat V_\ell Z(\,:\,,1:k)$}
\STATE{\textbf{return} $\widetilde\vx$, $U$, $SU$, $SAU$}
\smallskip
\end{algorithmic}
\end{algorithm}

\subsection{Krylov basis generation} In order to generate the Krylov bases $V_m$ at a cost that grows only linearly with $m$, we use the truncated Arnoldi process in Algorithm~\ref{alg:gmres_sdr}; \rev{see, e.g., \cite[Alg.~6.6]{Saad2003}.} This means that each new basis vector is projected against the previous $t\ll m$ vectors only.

\subsection{Solution of the projected least squares problem} In standard GMRES the projected least squares problem is usually solved by updating a QR factorization of the Hessenberg matrix $\underline{H_m}$ using Givens rotations. Here, the matrix $SA\vV_m$ in the least squares problem \eqref{eq:sgmres} is not of Hessenberg form. However, the matrix $\vV_m = [U,V_m]$ grows one column at a time when the truncated Arnoldi process for generating~$V_m$ progresses, hence we can still update the QR factorization of $SA\vV_m$ when a column is added, e.g., by using Householder-QR or the modified Gram--Schmidt with reorthogonalization.

\subsection{Residual control and updating \texttt{safety}}\label{sec:control-and-safety} 

The residual of the sketched problem \eqref{eq:sgmres} can serve as an approximation for the true residual of the current approximant using 
$$ {\|\widetilde \vr_m\|} \leq \frac{1}{\sqrt{1-\varepsilon}}  \|S \widetilde \vr_m \|. $$
However, we generally do not know what $\varepsilon$ is. Hence we propose the following strategy: within the GMRES-SDR cycle we monitor the sketched residual norm $\|S \widetilde \vr_m \|$. Only if this norm satisfies 
\[
\|S \widetilde \vr_m \| < \texttt{tol} / \texttt{safety}
\]
for some safety factor $\texttt{safety}>1$, we compute the true residual norm $\| \widetilde \vr_m\|$. If also this quantify is smaller than $\texttt{tol}$, we can terminate the Arnoldi iteration. Otherwise we will increase the safety factor to
\[
\texttt{safety} = \| \widetilde \vr_m\| / \| S\widetilde \vr_m\|
\]
and continue. 

An example illustration is given in \Cref{fig:sio2_distort.pdf}. The restart length in this example is chosen as $m=100$, and it is interesting to note that the \rev{distortion $\|\widetilde \vr_j\| / \|S \widetilde \vr_j\|$ does not look completely random, but appears to locally follow upward and downward trends within each Arnoldi cycle. We also notice a spike in the measured distortion of every $m$th basis vector, $\| \vv_{km} \|/\|S \vv_{km}\| = 1/\|S \vv_{km}\|$, with the height of the spike being very close to the distortion on residual. This is likely caused by the GMRES property that the Krylov basis vector  at the beginning of a cycle is strongly correlated with the residual $ \widetilde\vr_{km}$ from the previous cycle (in FOM these vectors would be perfectly collinear), and the fact that $\|S \widetilde\vr_{km}\|$ is minimized over the current subspace by sketched GMRES.} 
In our experiments we found that an initial value $\texttt{safety}=1.4$ works well, with the computation of the true residual only needed very rarely (apart from the unavoidable computation at the end of each Arnoldi cycle). 

\begin{figure}
    \centering
    \includegraphics[width=0.8\textwidth]{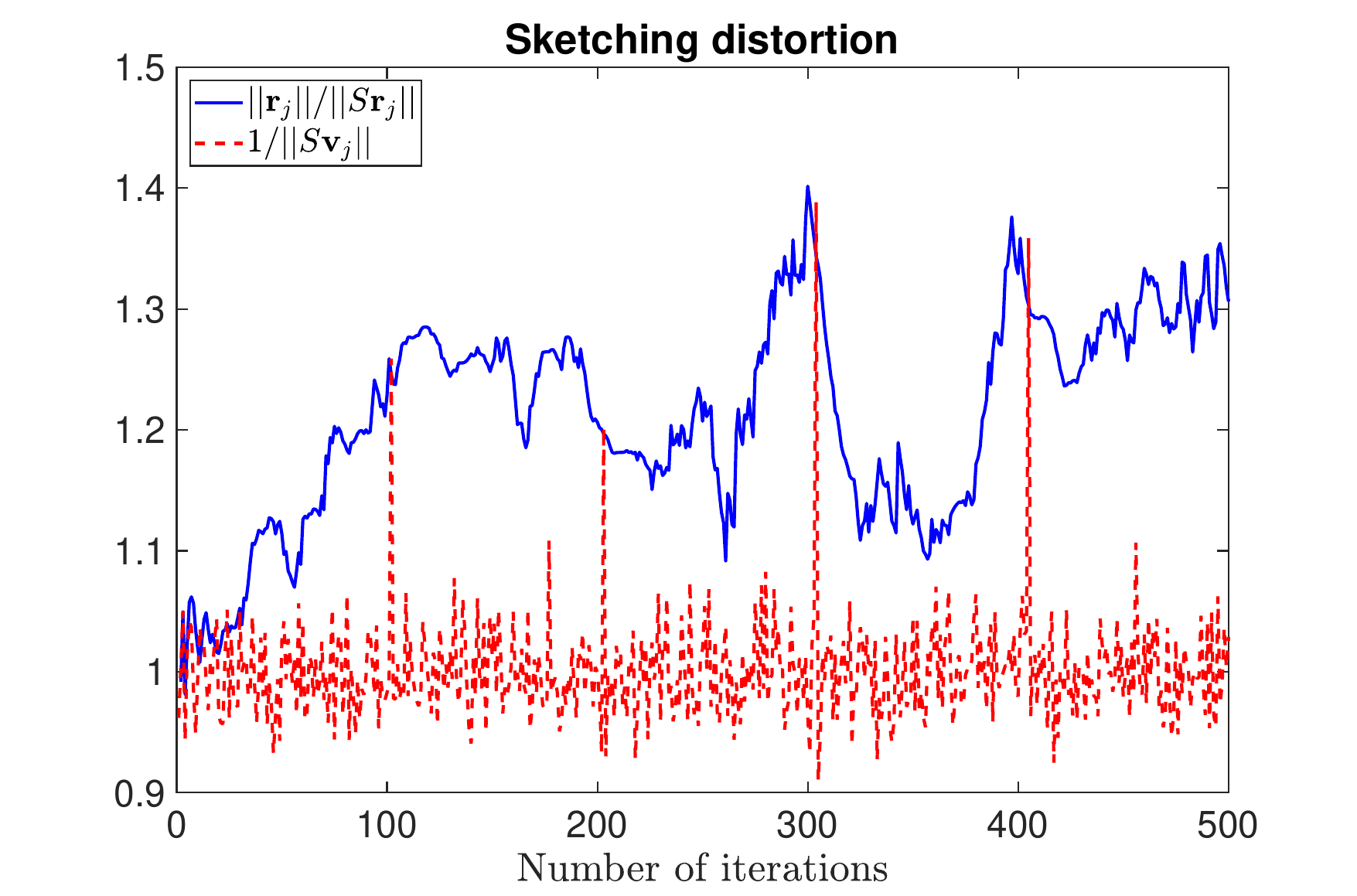}
    \caption{Example showing measured sketching distortion factors as the GMRES-SDR iteration progresses, as well as the ratio between the true and sketched residuals. In this example, the Krylov dimension of size $m = 100$, a truncation parameter of $t=2$, a recycling subspace dimension of size $k=20$, and a sketching parameter $s = 500$ was used. }
    \label{fig:sio2_distort.pdf}
\end{figure}
 
\subsection{Harmonic updating of the augmentation space}\label{sec:sRRharm}

\rev{We now turn our attention to the problem of updating the augmentation space $U$ from one restart cycle or problem to the next. 
Let the columns of the matrix $\vV_m$ span some subspace, and assume that we have the sketches $S \vV_m$ and $S A \vV_m$ at our disposal. 
The (non-sketched) harmonic Ritz pairs~\cite{PaigeEtAl1995} of the matrix A with respect to the subspace $\spn (\vV_{m})$} are defined as $(\vartheta,\vz=\vV_m \vy_m)$ such that
\[
    A\vz - \vartheta \vz \perp A \vV_m, \quad \rev{\vz \neq 0.}
\]
Therefore,
\[
(A\vV_m)^* ( A\vV_m \vy_m - \vartheta \vV_m \vy_m ) = 0,
\]
which means that
\[
(A\vV_m)^* (A\vV_m) \vy_m = \vartheta (A\vV_m)^* \vV_m \vy_m,
\]
or equivalently,
\[
 \vy_m = \vartheta [(A\vV_m)^* (A\vV_m)]^{-1} (A\vV_m)^* \vV_m \vy_m = \vartheta (A\vV_m)^\dagger \vV_m \vy_m.
\]
In other words, the harmonic Ritz pairs can be computed by solving an eigenvalue problem 
\[
M\vy_m := [(A\vV_m)^\dagger \vV_m] \vy_m = \vartheta^{-1} \vy_m.
\]

Now to utilize sketching, we realize that each column of the matrix $M$ can be thought of as the solution of a least squares problem with the tall skinny matrix~$A\vV_m$.

\rev{Or more explicitly, we define the sketched harmonic Ritz pairs of the matrix $A$ with respect to $\spn(\vV_{m})$ as $(\vartheta,\vz=\vV_m \vy_m)$ such that
\[
    S (A\vz - \vartheta \vz) \perp S A \vV_m, \quad \rev{\vz \neq 0.}
\] 
Therefore,
\[
(S A\vV_m)^* ( S A\vV_m \vy_m - \vartheta S \vV_m \vy_m ) = 0,
\]
which, as before, we can equivalently write as,
\[
 \vy_m = \vartheta [(S A\vV_m)^* (S A\vV_m)]^{-1} (S A\vV_m)^* S \vV_m \vy_m = \vartheta (S A\vV_m)^\dagger S \vV_m \vy_m.
\]
Thus, the sketched eigenproblem is}
\[
\widehat M \widehat\vy_m := [(SA\vV_m)^\dagger S\vV_m] \widehat\vy_m = \widehat\vartheta^{-1} \widehat\vy_m.
\]
Note that $\widehat M$ is the solution (of minimum Frobenius norm) to 
\begin{equation}
    \widehat M := \arg \min_{M\in\mathbb{C}^{(m+k)\times (m+k)}} \| S A \vV_m M - S\vV_m  \|_F,
    \label{eq:srr}
\end{equation}
a harmonic version of the sRR approach presented in \cite[Sec.~6.3]{NakatsukasaTropp21}. 

\rev{In the context of GMRES-SDR, assuming we have computed a Krylov basis $V_m$ of $\mathcal{K}_m(A,\vr_0)$ in one cycle, and we have an augmentation basis $U\in\mathbb{C}^{N\times k}$ accumulated over previous cycles, we then update the augmentation subspace $U$ using this sketched harmonic Ritz problem with 
$\vV_{m} = [U, V_{m}]$.   
}

To stabilize the harmonic Ritz extraction, we exploit a truncated SVD $S A \vV_m \approx \widehat U_\ell \Sigma_\ell \widehat V_\ell^*$ and replace
\[
S A \vV_m \rightarrow  \widehat U_\ell \quad \text{and} \quad S\vV_m \rightarrow (S\vV_m) \widehat V_\ell \Sigma_\ell^{-1},
\]
leading to the regularized 
\[
  \widehat M_\ell := \widehat U_\ell^* (S\vV_m) \widehat V_\ell \Sigma_\ell^{-1} \approx \widehat M. 
\]
We then compute an ordered QZ factorization such that
\[
Q [\widehat U_\ell^* (S \vV_m) \widehat V_\ell ] Z   = Q [ \Sigma_\ell ] Z 
\]
is upper triangular with the $k$ most relevant generalized eigenvalues appearing in the top-left block. As we are interested in harmonic Ritz values, we should target the \emph{largest modulus} eigenvalues. The desired harmonic Ritz vectors are then given as
\[
    U: = [U,V_m ] \widehat V_\ell Z(\,:\,,1:k).
\]
Similarly, the sketched version of $U$ and $AU$ can be computed without explicit sketches as shown in Algorithm~\ref{alg:gmres_sdr}.

\subsection{Solving slowly varying linear systems}\label{sec:slow} So far, we have introduced GMRES-SDR for a single linear system $A\vx = \vb$. Let us now consider the cases that $A\vx^{(i)} = \vb^{(i)}$ with varying right-hand side vectors, and $A^{(i)}\vx^{(i)} = \vb^{(i)}$ with varying vectors and system matrices.

In the first case, when only $\vb^{(i)}$ varies, we can simply reuse the recycling subspace $U, SU, SAU$ from one problem to the next. We just take the outputs $U, SU, SAU$ from Algorithm~\ref{alg:gmres_sdr} computed during the solution of $A\vx^{(i)} = \vb^{(i)}$ and feed them as inputs to Algorithm~\ref{alg:gmres_sdr} when solving $A\vx^{(i+1)} = \vb^{(i+1)}$.

The second case with problem-dependent system matrices, $A^{(i)}\vx^{(i)} = \vb^{(i)}$, allows for two variants which we refer to \emph{exact} and \emph{inexact}  GMRES-SDR, respectively. In the exact variant we reuse $U, SU$ from one problem $i$ to problem $i+1$, but compute $S A^{(i+1)} U$ by explicit matrix multiplication and sketching. This requires $k$ additional matrix-vector products and the sketching of $k$ additional vectors for each new problem. 

In the inexact GMRES-SDR variant we  feed $SA^{(i)} U$ as an input to Algorithm~\ref{alg:gmres_sdr} even though we are now solving $A^{(i+1)}\vx^{(i+1)} = \vb^{(i+1)}$ with a different system matrix. As a consequence, the matrix $SAW$ computed in line~18 of Algorithm~\ref{alg:gmres_sdr} becomes
\[
SAW = [S A^{(i)} U, SA^{(i+1)} V(:,1:j)],
\]
i.e., we are \rev{mixing (sketched) subspaces associated with with different matrices $A^{(i)}$ and $A^{(i+1)}$.} In particular, the least squares problem in line~21 no longer corresponds to solving a sketched version of the full GMRES problem for $A^{(i+1)}\vx^{(i+1)} = \vb^{(i+1)}$ and the residual relations~\eqref{eq:reschain} are violated. Still, we have observed numerically that one can sometimes get away with the inexactness when $A^{(i)}$ changes only very slightly from one problem to the next. However, there is no guarantee that this works and the sketched residuals may even increase as the iteration progresses. We will demonstrate this in \Cref{sec:cd}.

\section{Analysis of sketched GMRES}\label{section:sketched-gmres-analysis}
We build on the one-cycle description of sketched GMRES as introduced in~\cite{NakatsukasaTropp21}.
    They describe a general matrix $ V_m\in\mathbb{C}^{N\times m}$  whose columns are some basis vectors of a Krylov subspace, i.e., $\rng(V_m) = \mathcal{K}_m(A,\vr_0)$.  We extend existing GMRES convergence theory to the sketched GMRES setting, separating into basis-agnostic and basis-dependent results. Then in \Cref{section:sketched-aug-gmres-analysis} we put sketched augmented GMRES into the existing framework of augmented iterative methods.  This allows us to extend the analysis for sketched GMRES to augmented sketched GMRES.

    \subsection{Sketching the constraint space for one cycle of sketched GMRES}
    We pick up from the discussion in \Cref{section:GMRES-classic} and take a similar approach to understand the constraint for sketched GMRES. For the same basis $ V_m$, the sketched GMRES approximation is defined as
    \begin{align*}
        \widetilde{\vx}_m =  \vx_0 +  V_m\widetilde{\vy}_m, 
    \end{align*}
    where $\widetilde{\vy}_m$ minimizes $\left\|  S( A V_m\vy - \vr_0) \right\|$, 
    and via the normal equations, it follows that 
    \begin{align*}
        \widetilde{\vy}_m
        \rev{
        =
        \left(S A V_m\right)^\dagger S \vr_0
        }
        =
        \left(  V_m^\ast A^\ast S^\ast S A V_m \right)^{\dagger}
        \left(  V_m^\ast A^\ast S^\ast S \right)\vr_0.
    \end{align*}
    \rev{
    We note that if the columns of $S A V_m$ are linearly independent, then the associated Gram matrix is invertible; i.e., 
    \begin{align*}
        \left(  V_m^\ast A^\ast S^\ast S A V_m \right)^{\dagger}
        =
        \left(  V_m^\ast A^\ast S^\ast S A V_m \right)^{-1}.
    \end{align*}
    We often assume this as a property of the bases we sketch throughout this paper.
    }
    The sketched GMRES residual satisfies the relation
    \begin{align*}
        \widetilde{\vr}_m 
        = &
        \rev{\vr_0 -  A V_m\left(S A V_m\right)^\dagger S \vr_0 }
        = 
        \left( I - \widehat{\Phi}_{\CK_m} \right)\vr_0, 
        \\
        \rev{
        \ \ \mbox{where}\ \ 
        }
        \widehat{\Phi}_{\CK_m} 
        =&  
        \rev{
        A V_m\left(S A V_m\right)^\dagger S.
        }
    \end{align*}
    Similarly, if we let 
    \begin{align*}
        \widehat{\Pi}_{\CK_m} 
        =  
        \rev{
        V_m\left(S A V_m\right)^\dagger S A,
        }
    \end{align*}
     the sketched GMRES approximation satisfies $\hat{\vx}_m = \vx_0 + \widehat{\Pi}_{\CK_m}\veta_0$.
    We observe that
    $ S\widehat{\Phi}_{\CK_m} = \Psi_{S A\CK_m} S$, where $\Psi_{S A\CK_m}$
    is the orthogonal projector onto $\rng\left(  S A V_m \right)$, meaning that the sketched residual satisfies $ S\hat{\vr}_m\perp\rng\left(  S A V_m \right)$.  This is easy to see since we solved a sketched least squares problem.

    \subsection{Sketched GMRES is a projection method} 
    
    The matrices $\widehat{\Pi}_{\CK_m}$ and $\widehat{\Phi}_{\CK_m}$ serve the same role as the error and residual projectors in GMRES.  We \rev{show that they are indeed} projectors.  This builds on ideas in, e.g., \cite{balabanov2019randomized,palitta2023sketched}.  Note that in the following discussion, the space spanned by the columns of $ S^\ast S A V_m$ is discussed.  Since $ S$ sketches the true GMRES constraint space, the space must be lifted or interpolated back up to the full space in order to be used as a sketched residual constraint.  The operator $ S^\ast$ lifts the sketched space back to a subspace in $n$ dimensions.
    \begin{lemma}
        The matrices $\widehat{\Pi}_{\CK_m}$ and $\widehat{\Phi}_{\CK_m}$ are \rev{both projectors}, with $\rnga{\widehat{\Phi}_{\CK_m}}\subseteq\rng( A V_m)$ with null space containing $\rng( S^\ast S A V_m)^\perp$ and $\rnga{\widehat{\Pi}_{\CK_m}}\subseteq\rng( V_m)$ with null space containing $\rng( A^\ast S^\ast S A V_m)^\perp$.
    \end{lemma}
    \begin{proof}
        Idemptoency can be proven directly.  For $\widehat{\Phi}_{\CK_m}$, we can write
        \rev{
        \begin{align*}
            \widehat{\Phi}_{\CK_m}^2 
            &=
             A V_m\left(S A V_m\right)^\dagger S
             A V_m\left(S A V_m\right)^\dagger S
            \\
            &=
             A V_m\left(S A V_m\right)^\dagger S.
        \end{align*}
        which follows from one of the defining properties of the Moore--Penrose pseudoinverse. %
        }%
        Similarly for $\widehat{\Pi}_{\CK_m}$,
        \rev{
        \begin{align*}
            \widehat{\Pi}_{\CK_m}^2
            &=
             V_m\left(S A V_m\right)^\dagger S A
             V_m\left(S A V_m\right)^\dagger S A
            \\
            &=
             V_m\left(S A V_m\right)^\dagger S A
            =
            \widehat{\Pi}_{\CK_m}.
        \end{align*}
        }
        The range and null space containment assertions follow directly from the definitions of $\widehat{\Pi}_{\CK_m}$ and $\widehat{\Phi}_{\CK_m}$.  However, we note that if 
        \rev{the columns of $S A V_m$ are not linearly independent so that}
        $\left(  V_m^\ast A^\ast S^\ast S A V_m \right)^{\dagger}$ is not a true inverse, then it has as its null space $\rng\left(  V_m^\ast A^\ast S^\ast S A V_m \right)^\perp$. Furthermore, $S$ itself has a  \rev{null space of dimension at least $N-s$.}  \rev{Thus, the ranges and null spaces of these projectors fluctuate depending on the choice of sketching operator $S$}.
    \end{proof}
    \subsection{Semi-inner product induced by $ S^\ast  S $}
    How shall we understand these projectors $\widehat{\Pi}_{\CK_m}$ and $\widehat{\Phi}_{\CK_m}$?  Consider the fact that if $ S^\ast  S$ were nonsingular, it would induce an inner product.  In this case, one can easily show that these could be characterized with  $\widehat{\Pi}_{\CK_m}$ being an orthogonal projector onto $\rng( V_m)$ with respect to the $ A^\ast S^\ast S A$-inner-product, and $\widehat{\Phi}_{\CK_m}$ being an orthogonal projector onto $\rng( A V_m)$ with respect to the $ S^\ast S$-inner-product. This would simply be the method called often in the literature \emph{weighted GMRES}\cite{Essai:1998:weighted-gmres}.

    However, $ S^\ast S$ is \emph{necessarily} singular in the case of sketching.  It is thus pointed out in \cite{balabanov2019randomized} that it induces a nonnegative bilinear form satisfying all the axioms of inner products except that $\vw^\ast S^\ast S\vw = 0$ does not necessarily imply that $\vw=\mathbf{0}$.  Of course, because of the nature of the sketching process, this is not surprising.  Thus, sketched GMRES may be understood as an extreme version of weighted GMRES \cite{Essai:1998:weighted-gmres} wherein most of the weights \rev{(i.e., the eigenvalues of the matrix $S^\ast S$ inducing the bilinear form)} are actually zero, and the nonzero weights have been chosen such that the weighted semi-inner product still acts as an inner product on the correction and constraint spaces. 

    It is observed in \cite{balabanov2019randomized} that the bilinear form induced by positive semi-definite $ S^\ast S$ is called a \emph{semi-inner product} or pseudo-inner product.  Let us denote the semi-inner product by $\langle \vu,\vw \rangle_{ S^\ast S} := \vw^\ast S^\ast S\vu$.  In \cite{palitta2023sketched}, the authors point out further that the sketching assumption \eqref{eq:sketch} means that this semi-inner product acts as an inner product on the subspace in question.

    \subsection{Relationship between residual angle and sketched constraint space}
    In \cite{EiermannErnst2001}, the authors explored many aspects of the geometric properties of quantities and subspaces generated by Krylov subspace iterations, in the context of understanding aspects of convergence.  We now extend some of these results to the sketched GMRES setting.  Much work is done concerning geometric aspects of minimum residual methods such as GMRES, and we focus mainly on extending the ideas used to develop \cite[Theorem 4.5]{EiermannErnst2001} and results subsequent to that.

    We denote by $\ang{\vv}{\vu}$ the angle between vectors $\vv$ and $\vu$, and by $\ang{\vv}{\mathcal{W}}$ the principal angle between the vector $\vv$ and the subspace $\mathcal{W}$.  The authors of \cite{EiermannErnst2001} make use of the observation (true for any orthogonal projector onto a subspace) that 
    \begin{align*}
        \left\| 
            \Phi_{\CK_m}\vr_0
        \right\|
        =
        \left\| 
            \vr_0
        \right\|
        \cos
        \ang{\vr_0}{ A\CK_m}
        \ 
        \mbox{and,}
        \ 
        \left\| 
            \left(
                I - \Phi_{\CK_m}
            \right)
            \vr_0
        \right\|
        =
        \left\| 
            \vr_0
        \right\|
        \sin
        \ang{\vr_0}{ A\CK_m},
    \end{align*}
    This identity does not necessarily hold for sketched projectors, but the residual projector $\Phi_{\CK_m}$ can be related to a true projector acting in the sketched constraint space.  This can be used to connect us back to the theory used in \cite[Section 4]{EiermannErnst2001}, when acting on vectors from the sketched subspace, specifically $\vr_0$.
    \begin{lemma}
        For sketched GMRES, it holds for the sketched residual projection 
        \begin{align}
            \left\| 
            \widetilde{\vr}_m
        \right\|
        _{ S^\ast S}
        =
            \left\| 
            \left(
                I - \widehat{\Phi}_{\CK_m}
            \right)
            \vr_0
        \right\|
        _{ S^\ast S}
        =
        \left\| 
             S\vr_0
        \right\|
        \sin
        \ang{ S\vr_0}{ S A\CK_m}.
        \label{eq:orth-proj-trig}
        \end{align}
    \end{lemma}
    \begin{proof}
        This follows directly from the fact that $ S\widehat{\Phi}_{\CK_m} = \Psi_{SA\CK_m} S$, where $\Psi_{SA\CK_m}$ is the orthogonal projector acting in the sketched space, projecting onto $ S A\CK_m( A,\vr_0)$.  From this the result follows directly since
        \begin{align*}
            \left\| 
             S
            \left(
                I - \widehat{\Phi}_{\CK_m}
            \right)
            \vr_0
        \right\|
        =
        \left\| 
            \left(
                I - \Psi_{SA\CK_m}
            \right)
             S\vr_0
        \right\|,
        \end{align*}
        to which we can then apply \eqref{eq:orth-proj-trig}.
    \end{proof}
    Notice that we can consider developing bounds on the true residual in terms of sketched angles.  To make use of this, we first relate trigonometric functions of sketched angles back to their unsketched counterparts.
    \rev{
    \begin{lemma}\label{lemma:sketched-trig-bounds}
        Let $\vu$ and $\vv$ satisfy the sketching assumption \eqref{eq:sketch} for $ S$.  Then it holds that
        \begin{align}
            \dfrac
            {\cos\ang{\vu}{\vv}-\varepsilon}
            {\left( 1+\varepsilon \right)}
            \leq&
            \cos\ang{ S\vu}{ S\vv}
            \leq
            \dfrac
            {\cos\ang{\vu}{\vv}+\varepsilon}
            {\left( 1-\varepsilon \right)},
            \qquad
            \mbox{and}
            \nonumber
            \\
            \dfrac{
                \sqrt{
                    \sin^2\ang{\vu}{\vv} 
                    -
                    2\varepsilon(1 - \cos\ang{\vu}{\vv})
                }
             }
             {
                (1-\varepsilon)
             }
             \leq&
             \sin\ang{ S\vu}{ S\vv}
            \leq
            \dfrac{
                \sqrt{
                    \sin^2\ang{\vu}{\vv} 
                    +
                    2\varepsilon(1 + \cos\ang{\vu}{\vv})
                }
            }
            {
                (1+\varepsilon)
            }.
            \label{eq:sketched-trig-bounds}
        \end{align}
    \end{lemma}
    \begin{proof}
        Dividing the inequality \eqref{eq:sketch_innerproduct} by $\left\|\vu\right\|\left\|\vv\right\|$ yields
        \begin{align}
            \cos\ang{\vu}{\vv}-\varepsilon
            \leq
            \dfrac
            {\left\| S\vu\right\|\left\| S\vv\right\|}
            {\left\|\vu\right\|\left\|\vv\right\|}
            \cos\ang{ S\vu}{ S\vv}
            \leq
            \cos\ang{\vu}{\vv}+\varepsilon.
            \label{eq:sketched-cos-ineq}
        \end{align}
        From the sketching assumption \eqref{eq:sketch}, we have the bound
        \begin{align}
            \dfrac
            {1}
            {\left( 1+\varepsilon \right)}
            \leq
            \dfrac
            {\left\|\vu\right\|\left\|\vv\right\|}
            {\left\| S\vu\right\|\left\| S\vv\right\|}
            \leq
            \dfrac
            {1}
            {\left( 1-\varepsilon \right)}
            \label{eq:recip-ineq-util}
        \end{align}
        Multiplying \eqref{eq:sketched-cos-ineq} through by 
        $
            \dfrac
            {\left\|\vu\right\|\left\|\vv\right\|}
            {\left\| S\vu\right\|\left\| S\vv\right\|}
        $
        and applying \eqref{eq:recip-ineq-util} yields the first result in~\eqref{eq:sketched-trig-bounds}.

        To obtain the second result, we take the first result, square it, negate it (reversing inequality order) and add $1$ to it, yielding
        \begin{align*}
            1
            -
            \left(
                \dfrac
                {\cos\ang{\vu}{\vv} + \varepsilon}
                {\left(
                    1-\varepsilon
                \right)}
            \right)
            ^2
            \leq
            \sin^2\ang{ S\vu}{ S\vv}
            \leq
            1
            -
            \left(
                \dfrac
                {\cos\ang{\vu}{\vv} - \varepsilon}
                {\left(
                    1+\varepsilon
                \right)}
            \right)
            ^2.
        \end{align*}
         Getting everything over a common denominator yields
         \begin{align*}
             \dfrac{
                (1-\varepsilon)^2 
                -
                (\cos\ang{\vu}{\vv} + \varepsilon)^2
             }
             {
                (1-\varepsilon)^2
             }
             \leq
             \sin^2\ang{ S\vu}{ S\vv}
            \leq
            \dfrac{
                (1+\varepsilon)^2 
                -
                (\cos\ang{\vu}{\vv} - \varepsilon)^2
            }
            {
                (1+\varepsilon)^2
            }
         \end{align*}
         and expanding the numerators out yields
         \begin{align*}
             \dfrac{
                \sin^2\ang{\vu}{\vv} 
                -
                2\varepsilon(1 - \cos\ang{\vu}{\vv})
             }
             {
                (1-\varepsilon)^2
             }
             \leq
             \sin^2\ang{ S\vu}{ S\vv}
            \leq
            \dfrac{
                \sin^2\ang{\vu}{\vv} 
                +
                2\varepsilon(1 + \cos\ang{\vu}{\vv})
            }
            {
                (1+\varepsilon)^2
            }
         \end{align*}
        Taking square roots yields the second result from \eqref{eq:sketched-trig-bounds}.
    \end{proof}
    }
    \cref{lemma:sketched-trig-bounds} directly leads to a bound on the sketched GMRES residual norm.
    \begin{corollary}\label{corollary:sketched-gmres-angle-bound}
        The sketched GMRES residual $\widetilde{\vr}_m$ satisfies the subspace angle-based bound
        \rev{
        \begin{align}
            \left\|
                \widetilde{\vr}_m
            \right\|
            \leq&
            \left\|
                \vr_0
            \right\|
            \sqrt{
                \dfrac
                {
                    \sin^2\ang{\vr_0}{ A\CK_m\left( A,\vr_0\right)} +2\varepsilon(1 + \cos\ang{\vr_0}{ A\CK_m\left( A,\vr_0\right)})
                }
                {
                    1-\varepsilon^2
                }
            }
            \nonumber
            \\
            \leq&
            \left\|
                \vr_0
            \right\|
            \sqrt{
                \dfrac
                {
                    \sin^2\ang{\vr_0}{ A\CK_m\left( A,\vr_0\right)} +4\varepsilon
                }
                {
                    1-\varepsilon^2
                }
            }.
            \label{eq:sketched-residual-bound-estimate}
        \end{align}
        }
    \end{corollary}
    \begin{proof}
        Squaring $\left\|\widetilde{\vr}_m \right\|_{S^*S}$ and combining
         \eqref{eq:orth-proj-trig} with the left side of the sketching assumption \eqref{eq:sketch} for $\widetilde{\vr}_m$ yields 
         \rev{
        \begin{align*}
            \left(
                1-\varepsilon
            \right) 
            \left\|
                \widetilde{\vr}_m
            \right\|^2
            \leq
            \left\| 
                \widetilde{\vr}_m 
            \right\|^2_{S^*S}
            \leq
                \left\|
                     S\vr_0
                \right\|^2
                \sin^2\ang{ S\vr_0}{ S A\CK_m\left( A,\vr_0\right)}.
        \end{align*}
        
        Dividing through by $1-\varepsilon$ and squaring yields 
        \begin{align*} 
            \left\|
                \widetilde{\vr}_m
            \right\|^2
            \leq
            \dfrac{
                \left\|
                     S\vr_0
                \right\|^2
                \sin^2\ang{ S\vr_0}{ S A\CK_m\left( A,\vr_0\right)}
            }
            {
                (1-\varepsilon)
            }
        \end{align*}
        }%
        combining this with the right side of the sketching assumption \eqref{eq:sketch} and
        the right side of the second inequality in \eqref{eq:sketched-trig-bounds}  and taking square roots yields the first result.
        Bounding $\cos\ang{\vr_0}{ A\CK_m\left( A,\vr_0\right)}$ from above yields the second result.
    \end{proof}
    \rev{%
    We observe that the utility of this residual norm estimate varies with $\varepsilon$ and also with the value of the sine of the unsketched subspace angle $\sin\ang{\vr_0}{ A\CK_m\left( A,\vr_0\right)}$, which can be demonstrated by studying the relative difference between the first bound in \eqref{eq:sketched-residual-bound-estimate} and the true GMRES residual bound (i.e., when $\varepsilon=0$) for varying pairs $(\theta, \varepsilon)$; see \Cref{fig:subspace-angle-based-estimate}.
    In addition, we observe that this bound indicates that as true GMRES converges, we can expect sketched GMRES to similarly converge even for some larger values of $\varepsilon$.
    }%
    \begin{figure}[h]
        \centering
        \includegraphics[width=0.7\linewidth]{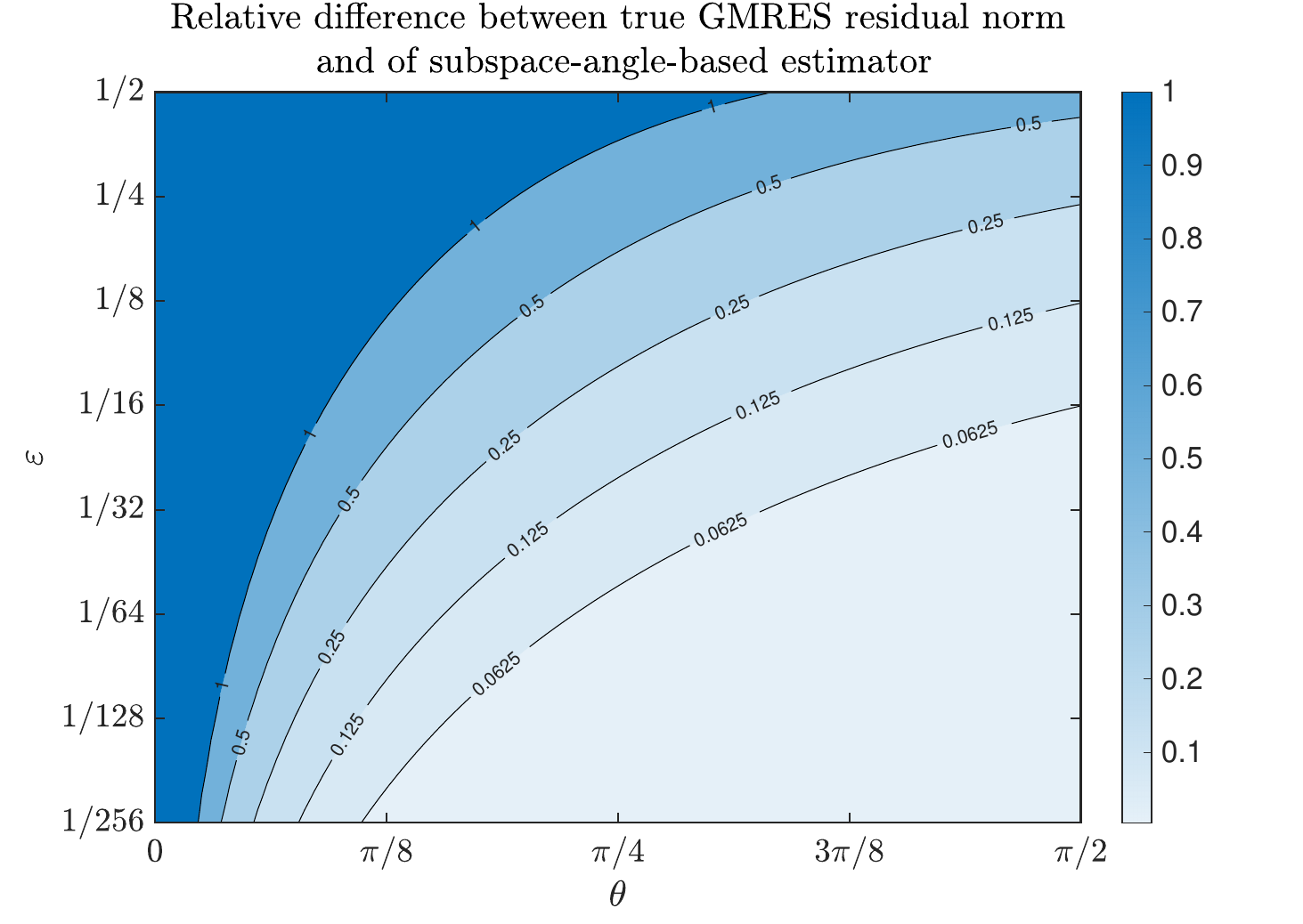}
        \caption{
        \rev{
        Contours of $d(\varepsilon, \theta)/\sin\theta$,
        the difference between the estimated sketched residual norm in the first bound from \eqref{eq:sketched-residual-bound-estimate}  
        and the true residual norm (for $\varepsilon=0)$, (relative to the size of $\sin\theta$); i.e., 
        $
        d(\varepsilon, \theta)
        =
        \sqrt{
            \dfrac
            {
                \sin^2\theta +2\varepsilon(1 + \cos\theta)
            }
            {
                1-\varepsilon^2
            }
        }
        -
        \sin\theta
        $.
        Contours are generated for $d(\varepsilon, \theta)/\sin\theta\in\left\lbrace 1/16, 1/8, 1/4, 1/2, 1\right\rbrace$, which should be understood as percentage relative difference.
        }
        }
        \label{fig:subspace-angle-based-estimate}
    \end{figure}
    \subsection{Analysis based on the sketched Arnoldi relation}
    We note that our basis $V_m$ for $\CK_m( A,\vr_0)$ satisfies 
    \begin{align}
        SA V_m
        =  
        SV_{m+1} \underline{H_m}
        \label{eq:truncated-sketched-Arnoldi}
    \end{align}
    where $\underline{H_m}$ is constructed progressively using a truncated Arnoldi orthogonalization with truncation parameter $t$; thus, is $t$-banded. Computing the economy QR decomposition $S V_{m+1}=\widehat{W}_{m+1}R_{m+1}$ and setting $\underline{\widehat{H}_m} = R_{m+1}\underline{H_m}$ (which is full upper Hessenberg) yields the sketched Arnoldi relation
    \begin{align}
        SA V_m 
        = 
        \widehat{W}_{m+1} \underline{\widehat{H}_m}.
        \label{eq:sketched-Arnoldi}
    \end{align}
    \rev{
    \begin{remark}
        Obtaining the full upper Hessenberg (i.e., doing a full untruncated sketched Arnoldi) what is advocated for in, e.g., \cite{BalabanovGrigori22}, but is not the approach we take in this manuscript.  However, we note that sketched GMRES algorithms derived from \eqref{eq:sketched-Arnoldi} and \eqref{eq:truncated-sketched-Arnoldi} are mathematically equivalent.  Thus, we analyze the formulation of sketched GMRES derived from \eqref{eq:sketched-Arnoldi} to understand what theoretical behavior we can expect from our method.
    \end{remark}
    }

    Consider that the version of sketched GMRES derived from \eqref{eq:sketched-Arnoldi} minimizes the same functional,
    \begin{align}
        \widetilde{\vx}_m 
        &=
        \argmin{
            \vt\in\CK_m( A,\vr_0)
        }
        \norm{
             S\left(\vr_0 -  A\vt\right)
        }
        \\
        \nonumber
        \iff
        \widetilde{\vx}_m
        =
        V_m\widetilde{\vy}_m
        \ \ \mbox{where}\ \ 
        \widetilde{\vy}_m
        &=
        \argmin{
            \vy\in\mathbb{C}^m
        }
        \norm{
            \underline{\widehat H_m}\vy 
            -
            \norm{ S\vr_0}\ve_1
        }.
        \label{eq:sketched-gmres-arnoldi-minimization}
    \end{align}
    The functional that $\widetilde{\vy}_m$ minimizes is the standard GMRES functional, simply for the sketched problem.  \emph{This is important} because any analysis of GMRES with proofs proceeding from an analysis of the minimization of this functional can be mapped onto analogous results for sketched GMRES.  As with full GMRES, the minimization in \eqref{eq:sketched-gmres-arnoldi-minimization} can be solved by progressively QR-factorizing $\underline{\widehat{H}_m}$ using one Givens rotation per iteration, obtaining at step $m$, yielding $\underline{\widehat{H}_m} = {\underline{\widehat Q}_m}{\underline{\widehat R_m}}$.
    \begin{theorem}\label{theorem:givens-sine-sketched-gmres-angles}
        Let us write ${\underline{\widehat Q_m}} = \left(\underline{\hat{q}}_{ij}\right)\in\mathbb{C}^{(m+1)\times(m+1)}$.  Then it follows that 
        \begin{align*}
            \sin\ang{ S\vr_0}{ S A\CK_m( A,\vr_0)} 
            = 
            \left| \hat{q}_{m+1,1} \right|.
        \end{align*}
        Furthermore, it follows that
        \begin{align*}
            \sin\ang{ S\tilde{\vr}_{m-1}}{ S A\CK_m( A,\vr_0)} 
            =
            \hat{s}_m
        \end{align*}
        where $\hat{s}_m$ is the $m$-th Givens sine used for constructing the QR-factorization of $\underline{\widehat{H}_m}$.  The sketched residual norm satisfies $\norm{ S\tilde{\vr}_m}=|\hat{s}_m|\norm{ S\tilde{\vr}_{m-1}}$.
    \end{theorem}
    \begin{proof}
        We note that 
        \begin{align*}
            \ang{ S\vr_0}{ S A\CK_m( A,\vr_0)}
            =
            \ang{\ve_1}{\rnga{\underline{\widehat{H}_m}}}.
        \end{align*}
        The rest follows step-for-step the proof of \cite[Theorem 4.5]{EiermannErnst2001} and analysis thereafter.
    \end{proof}
    \rev{
        We note that this is a useful insight for residual estimation.  It allows one to cheaply and quickly estimate the sketched residual as the true residual is at most a factor $(1-\varepsilon)^{-1/2}$ larger.
        Going further, if one were to implement sketched GMRES using \eqref{eq:sketched-Arnoldi} following \cite{BalabanovGrigori22}, one could use the Givens sine at each iteration $m$ to estimate the residual using bounds such as in \Cref{lemma:sketched-trig-bounds} and \Cref{corollary:sketched-gmres-angle-bound}.  This is of \emph{theoretical interest} but would in actuality be too expensive to implement in the setting of the sketched GMRES variants being considered in this paper.
    }

    \begin{corollary}\label{corollary:sketched-gmres-Arnoldi-givensNormUpdate}
        For the implementation of sketched GMRES using \eqref{eq:sketched-Arnoldi} following \cite{BalabanovGrigori22}, the true residual norm admits the estimate satisfying the progressive updating scheme using the 
        \rev{
        sines of the unsketched subspace angles, namely
        \begin{align}
            \norm{\tilde{\vr}_m} 
            \leq 
            \sqrt{
            \dfrac
                {s_m^2 + 2\varepsilon(1+c_m)}
                {1-\varepsilon^2}
            }
            \norm{\tilde{\vr}_{m-1}},
            \label{eq:sketched-gmres-Arnoldi-givensNormUpdate}
        \end{align}
        where $s_m = \sin\ang{ \tilde{\vr}_{m-1}}{ A\CK_m( A,\vr_0)} $ and $c_m = \cos\ang{ \tilde{\vr}_{m-1}}{ A\CK_m( A,\vr_0)} $.
        }
    \end{corollary}
    \begin{proof}
        This estimate follows from combining the results in \Cref{theorem:givens-sine-sketched-gmres-angles} with the sketching assumption \eqref{eq:sketch} and the inequality relating sines of sketched angles to sines of the unsketched analogs \eqref{eq:sketched-trig-bounds}.  With a bit of algebra, this yields the result.
    \end{proof}
    
    This sort of relationship is useful in the case that we implement sketched (augmented) GMRES using \eqref{eq:sketched-Arnoldi} following \cite{BalabanovGrigori22}.  \rev{However, we are not in that setting, and using \eqref{eq:sketched-gmres-Arnoldi-givensNormUpdate} would incur a nontrivial expense, as it would undo the savings attained from using truncation}.  


\subsection{The relationship between sketched GMRES and sketched FOM}
Interestingly, it also follows from \eqref{eq:sketched-Arnoldi} to derive sketched GMRES that we can show that the well-known relationship between true GMRES and true FOM can be extended to the sketched counterparts.  However, the result is actually true for any mathematically equivalent implementations of sketched GMRES and sketched FOM~\cite{BG23}.  

For clarity, we identify quantities associated with sketched GMRES using the superscript~$\cdot^{(G)}$ and quantities associated with sketched FOM using the superscript~$\cdot^{(F)}$.
By sketched FOM, we simply mean that we apply a sketched version of the FOM constraint to define the FOM iterate $\tilde{\vx}_m^{(F)} = \vx_0 + \tilde{\vt}_m^{(F)}$ via
\begin{align*}
    \mbox{Select}\ \ \tilde{\vt}_m^{(F)}\in\CK( A,\vr_0)
    \ \ \mbox{such that}\ \ 
     S\left(\vb- A\left(\vx_0 + \tilde{\vt}_m^{(F)}\right)\right) 
    \perp 
     S \CK_m( A,\vr_0).
\end{align*}
For $\tilde{\vt}_m^{(F)} = V_m\tilde{\vy}_m^{(F)}$ where $\tilde{\vy}_m^{(F)}$ satisfies the FOM condition 
\begin{align*}
    \widehat{H}_m \tilde{\vy}_m^{(F)} 
    = 
    \norm{ S\vr_0}\ve_1,
\end{align*}
where $\widehat{H}_m$ is the square matrix formed by the first $m$ rows of ${\underline{\widehat H_m}}$.  This means that sketched GMRES and sketched FOM form an MR/OR pair, as described in, e.g., \cite[Section 1]{EiermannErnst2001}.  We omit extending every such result, but we point the reader to the results in \cite{Brown1991}, which are proven by comparing the QR-factorizations of the square and rectangular upper Hessenberg matrices. In particular, following the derivation of \cite[eq.~6.74]{Saad2003}, a comparison of the QR-factorizations of ${\underline{\widehat H_m}}$ and $\widehat{H}_m$ shows the following. 
\begin{theorem}
    The sketched GMRES and sketched FOM iterates have the same relationship as the true GMRES and FOM iterates, namely
    \begin{align*}
        \tilde{\vx}_m^{(G)}
        =
        c_m^2\tilde{\vx}_m^{(F)}
        +
        s_m^2\tilde{\vx}_{m-1}^{(G)}.
    \end{align*}    
\end{theorem}

\smallskip

\section{Analysis of sketched augmented GMRES}\label{section:sketched-aug-gmres-analysis}
    \rev{%
    We have established that the sketched GMRES method is a projection method when considered on the correct space, with the projectors involving the embedding operator $S$.  
    }%
    In the context of the characterization of augmented residual minimization methods from \cite{Soodhalter:deSturler:Kilmer.2020.framework} that we summarize in \Cref{section:augmentation-framework}, we may ask: does a sketched augmented minimization method fit into a generalization of the framework from  \cite{Soodhalter:deSturler:Kilmer.2020.framework}?

    Consider the sketched augmented GMRES approximation $\widetilde{\vx}_m = \vx_0 + U\widetilde{\vz}_m + V_m\widetilde{\vy}_m$.  The vectors $\widetilde{\vz}_m$ and $\widetilde{\vy}_m$ are components of a solution to the normal equations
    \begin{align}
        \begin{bmatrix}
            U^\ast A^\ast S^\ast S A U
            &
            U^\ast A^\ast S^\ast S A V_m
            \\
            V_m^\ast A^\ast S^\ast S A U
            &
            V_m^\ast A^\ast S^\ast S A V_m
        \end{bmatrix}
        \begin{bmatrix}
            \widetilde{\vz}
            \\
            \widetilde{\vy}
        \end{bmatrix}
        &=
        \begin{bmatrix}
             U^\ast A^\ast S^\ast S\vr_0
             \\
              V_m^\ast A^\ast S^\ast S\vr_0
        \end{bmatrix}.
        \label{eq:block-normal-equations}
    \end{align}
    Because this is a sketched minimum residual method, these normal equations' coefficient matrix may have a null space; thus there could be multiple solutions.  By choosing 
    \begin{align*}
        \begin{bmatrix}
            \widetilde{\vz}_m
            \\
            \widetilde{\vy}_m
        \end{bmatrix}
        =
        \begin{bmatrix}
            U^\ast A^\ast S^\ast S A U
            &
            U^\ast A^\ast S^\ast S A V_m
            \\
            V_m^\ast A^\ast S^\ast S A U
            &
            V_m^\ast A^\ast S^\ast S A V_m
        \end{bmatrix}
        ^\dagger
        \begin{bmatrix}
             U^\ast A^\ast S^\ast S\vr_0
             \\
              V_m^\ast A^\ast S^\ast S\vr_0
        \end{bmatrix},
    \end{align*}
    the Moore--Penrose pseudoinverse solution, we obtain the solution which has minimum norm, containing no null space components. 
    
    In \cite{Soodhalter:deSturler:Kilmer.2020.framework}, the authors follow derivation first pointed out in \cite{ParksSoodhalterSzyld:2016:1} which uses block Gaussian elimination to eliminate $\widetilde{\vz}_m$ from the second block row of equations, yielding the result.  In the case that the matrix has a null space, general theory for when block Gaussian elimination can be applied to develop a generalized Schur complement has been derived; see, e.g., \cite{CarlsonHaynsworthMarkham:1974:gen-schur}.
    The case discussed in the present work is simpler.
    \begin{corollary}
        \rev{If the columns of $S A \vV_m$ are linearly independent}, the coefficient matrix in \eqref{eq:block-normal-equations} admits a Schur complement via elimination of the $(1,2)$ block using the block row operation
        \begin{align*}
            \begin{bmatrix}
                I
                &
                \rev{\mathbf{0}} \\
                -
                V_m^\ast A^\ast S^\ast S A U
                \left(
                    U^\ast A^\ast S^\ast S A U
                \right)^{-1}
                &
                I
            \end{bmatrix}
            \begin{bmatrix}
                U^\ast A^\ast S^\ast S A U
                &
                U^\ast A^\ast S^\ast S A V_m
                \\
                V_m^\ast A^\ast S^\ast S A U
                &
                V_m^\ast A^\ast S^\ast S A V_m
            \end{bmatrix},
        \end{align*}
        leading to the equivalent system
        \begin{align}
            \begin{bmatrix}
                U^\ast A^\ast S^\ast S A U
                &
                U^\ast A^\ast S^\ast S A V_m
                \\
                \mathbf{0}
                &
                V_m^\ast A^\ast S^\ast S
                \left(
                    I - \widehat{\Phi}_U
                \right)
                 A V_m.
            \end{bmatrix}
            \begin{bmatrix}
                \widetilde{\vz}
                \\
                \widetilde{\vy}
            \end{bmatrix}
            &=
            \begin{bmatrix}
                U^\ast A^\ast S^\ast S\vr_0
                \\
                V_m^\ast A^\ast S^\ast S
                \left(
                    I - \widehat{\Phi}_U
                \right)
                \vr_0
            \end{bmatrix}.
            \label{eq:block-ne-GE}
        \end{align}
    \end{corollary}
    Thus, as in \cite{Soodhalter:deSturler:Kilmer.2020.framework}, we have isolated in the second block of equations the variables $\widetilde{\vy}$ associated to the space $\rnga{V_m}$ that grows at each iterations.  We can prove the following theorem.
    \begin{theorem}\label{theorem:sketched-aug-gmres-subprob-equiv}
        For $U$ such that $\rnga{ A U}$ satisfies the sketching assumption \eqref{eq:sketch}, and for $m$ not too large, \rev{if the columns of $S A \vV_m$ are linearly independent}, applying augmented sketched GMRES for the augmented space $\rnga{\vV_m}$ is equivalent to applying a sketched GMRES iteration to the equation
        \begin{align}
            \left(
                I - \widehat{\Phi}_\vU
            \right)
             A\vt
            =
            \left(
                I - \widehat{\Phi}_\vU
            \right)
            \vr_0,
            \label{eq:sketched-aug-gmres-proj-subprob}
        \end{align}
        selecting $\vt_m=V_m\widetilde{\vy}_m\in\rnga{V_m}$ according to 
        \begin{align}
            \widetilde{\vy}_m
            =
            \underset
            {\vy\in\mathbb{C}^m}
            {\mathrm{argmin}}
            \left\| 
                 S
                \left(
                I - \widehat{\Phi}_\vU
                \right)
                \left(
                    \vr_0 -  A V_m \vy 
                \right)
                \right\|
                \label{eq:sketched-pseudoprojected-minimization}
        \end{align}
        and constructing $\widetilde{\vx}_m = \vx_0 + \widehat{\Pi}_U\veta_0 + (I-\widehat{\Pi}_U)\widetilde{\vt}_m$.
    \end{theorem}
    \begin{proof}
        This has mostly already been proven in the above discussion.  We observe only that the second block of equations in \eqref{eq:block-ne-GE} defines $\widetilde{\vy}$ to be the exactly the minimizer described in \eqref{eq:sketched-pseudoprojected-minimization}.  Inserting $\widetilde{\vy}_m$ into the first block of equations yields
        \begin{align*}
            \widetilde{\vz}
            &=
            U^\ast A^\ast S^\ast S\vr_0
            -
            U^\ast A^\ast S^\ast S A V_m
            \widetilde{\vy}_m
            \\
            \iff
            U \widetilde{\vz}
            &=
            \widehat{\Pi}_U\veta_0
            -
            \widehat{\Pi}_U
            \vt_m,
        \end{align*}
        completing the proof.
    \end{proof}
    It follows directly that the full augmented sketched GMRES residual is the same as that of the projected subproblem, meaning that the residual behavior of this method is completely dictated by the properties of the subproblem and the space $\rnga{V_m}$.
    \begin{corollary}\label{cor:projected-full-resid-equiv}
        Let $U$ and $AU$ satisfy the same conditions as in \Cref{theorem:sketched-aug-gmres-subprob-equiv}. Then it follows the full sketched augmented GMRES residual $\widetilde{\vr}_m = \vb - A\widetilde{\vx}_m$ is the same as the residual from approximately solving \eqref{eq:sketched-aug-gmres-proj-subprob} via \eqref{eq:sketched-pseudoprojected-minimization}; i.e.,
        \begin{align*}
            \widetilde{\vr}_m
            =
            \left(
                I - \widehat{\Phi}_\vU
            \right)
            \left(
                \vr_0 -  A \vt_m 
            \right)
        \end{align*}
    \end{corollary}
    \begin{proof}
        We compute
        \begin{align*}
            \widetilde{\vr}_m
            &=
            \vb - A\widetilde{\vx}_m
            \\
            &=
            \vb - A(\vx_0 + \widehat{\Pi}_U\veta_0 + (I-\widehat{\Pi}_U)\widetilde{\vt}_m)
            \\
            &=
            \vr_0 - \widehat{\Phi}_U\vr_0 - (I-\widehat{\Phi}_U)A\widetilde{\vt}_m),
        \end{align*}
        where we have taken advantage of the fact that $A\widehat{\Pi}_U = \widehat{\Phi}_U A$.
    \end{proof}

    \begin{remark}
        We make the observation that, were we to build $V_m$ such that 
        \begin{align*}
            \rnga{V_m} 
            = 
            \CK_m
            \left(  
                \left(
                    I - \widehat{\Phi}_\vU
                \right)
                A
                ,
                \left(
                    I - \widehat{\Phi}_\vU
                \right)
                \vr_0
            \right),
        \end{align*}
        we could use \Cref{cor:projected-full-resid-equiv} to exploit the results from \Cref{theorem:givens-sine-sketched-gmres-angles} and thereafter directly, as the method could directly be interpreted as a fully sketched version of GCRO-DR.   \Cref{cor:projected-full-resid-equiv} then implies that the residual convergence behavior of the iterative method is equivalent to that of sketched GMRES applied to \eqref{eq:proj-subprobem}, and we can estimate the residual norm using the Givens sines, as described in \Cref{theorem:givens-sine-sketched-gmres-angles}.  \rev{We are not considering a fully sketched version of GCRO-DR in this paper, but it would be interesting to consider our analysis applied to a sketched GCRO-DR.}
    \end{remark}

    We have demonstrated that we can gain computational advantage by forgoing the building of a sketched projected Krylov subspace.  Instead, we have
    \begin{align*}
        \rnga{V_m} 
        = 
        \CK_m
        \left( 
            A, \vr_0
        \right).
    \end{align*}    
    The price for this advantage is that \Cref{cor:projected-full-resid-equiv} no longer implies that the iterative method is equivalent to a sketched GMRES iteration applied to \eqref{eq:proj-subprobem}.  However, all is not lost; we can still use \Cref{theorem:givens-sine-sketched-gmres-angles}.  Consider that via truncated sketched Arnoldi, we have the relation 
    \begin{align*}
        SA\vV_m 
        =
        \begin{bmatrix}
            SAU
            &
            SV_{m+1}
        \end{bmatrix}
        \begin{bmatrix}
            I & 
            \\
             & \underline{H_m}\,
        \end{bmatrix}.
    \end{align*}
    We compute the economy QR-factorization
    $
        \begin{bmatrix}
            SAU
            &
            SV_{m+1}
        \end{bmatrix}
        =
        \widehat{\vW}_{m+1}
        R_m
    $,
    and block the upper triangular matrix
    \begin{align*}
        R_m
        =
        \begin{bmatrix}
            R_m^{(1,1)} & R_m^{(1,2)}
            \\
            & R_m^{(2,2)}
        \end{bmatrix},
        \qquad
        \mbox{with}
        \qquad
        R_m^{(1,1)} \in \mathbb{C}^{k\times k}.
    \end{align*}
    Then we can construct the upper-Hessenberg matrix
    \begin{align*}
    \underline{G_m} 
    = 
    R_m
    \begin{bmatrix}
        I & 
        \\
         & \underline{H_m}\,
    \end{bmatrix}
    =
    \begin{bmatrix}
        R_m^{(1,1)} & R_m^{(1,2)}\underline{H_m}
        \\
         & R_m^{(2,2)}\underline{H_m}\,
    \end{bmatrix}
    =
    \begin{bmatrix}
        R_m^{(1,1)} & \ast
        \\
         &\underline{\widehat{H}_m}
    \end{bmatrix},
    \end{align*}
    where $\underline{\widehat{H}_m} := R_m^{(2,2)}\underline{H_m}$.  It follows that we have the relation
    \begin{align*}
        SA\vV_m 
        =
        \widehat{\vW}_{m+1}
        \underline{G_m}.
    \end{align*}
    We observe that $\underline{G_m}$ is upper Hessenberg but with an already-triangular (1,1)-block.  This means that a Givens rotation-based QR factorization of $\underline{G_m}$ reduces to applying Givens rotations to $\underline{\widehat{H}_m}$.
    The proof of 
    \Cref{theorem:givens-sine-sketched-gmres-angles} allows us to conclude that
    \begin{align*}
        \ang
        {S\vr_0}
        {\rnga{SA\vV_m}} 
        = 
        \ang
        {\ve_1}
        {\rnga{\underline{G_m}}} 
        =
        \ang
        {\ve_1}
        {\rnga{\underline{Q_m}}},
    \end{align*}
    where $\underline{Q_m}\in\mathbb{C}^{(m+1)\times(m+1)}$ is the Q-factor of the full QR factorization of $\underline{G_m}$.
    Thus, because of the partially-triangularized structure of $\underline{G_m}$, we can use the same Givens rotation estimator strategy shown in \Cref{corollary:sketched-gmres-Arnoldi-givensNormUpdate}.  We demonstrate the effectiveness of the residual estimator for an example problem in \Cref{fig:resid-estimator.pdf}.    
    \begin{remark}
        Note that in \Cref{fig:resid-estimator.pdf}, we compute the estimate for every iteration for demonstrative purposes.  In reality, the frequency of the use of this estimate would depend on which implementation of augmented sketched GMRES.  If an orthonormal basis for $\rnga{SA\vV_m}$ is unavailable, the estimate would be used sparingly, due to the aforementioned expense of obtaining it.  However, we note that in \Cref{alg:gmres_sdr}, a basis-whitening technique is used which does produce an orthonormal basis.  In that case, Givens rotations could be computed progressively, one-per-iteration, in the usual way for GMRES. This would allow the estimate to be computed inexpensively.
    \end{remark}
    It is also of interest that given some true residual norm computations, one could use the bound from \Cref{corollary:sketched-gmres-Arnoldi-givensNormUpdate} to get an estimate for $\varepsilon$.  This is a possible alternative method to the method for estimation of $\varepsilon$ advocated in \Cref{sec:control-and-safety}, which is what is actually implemented in our proposed algorithm.

\begin{figure}[h]
    \centering
\includegraphics[width=.8\textwidth]{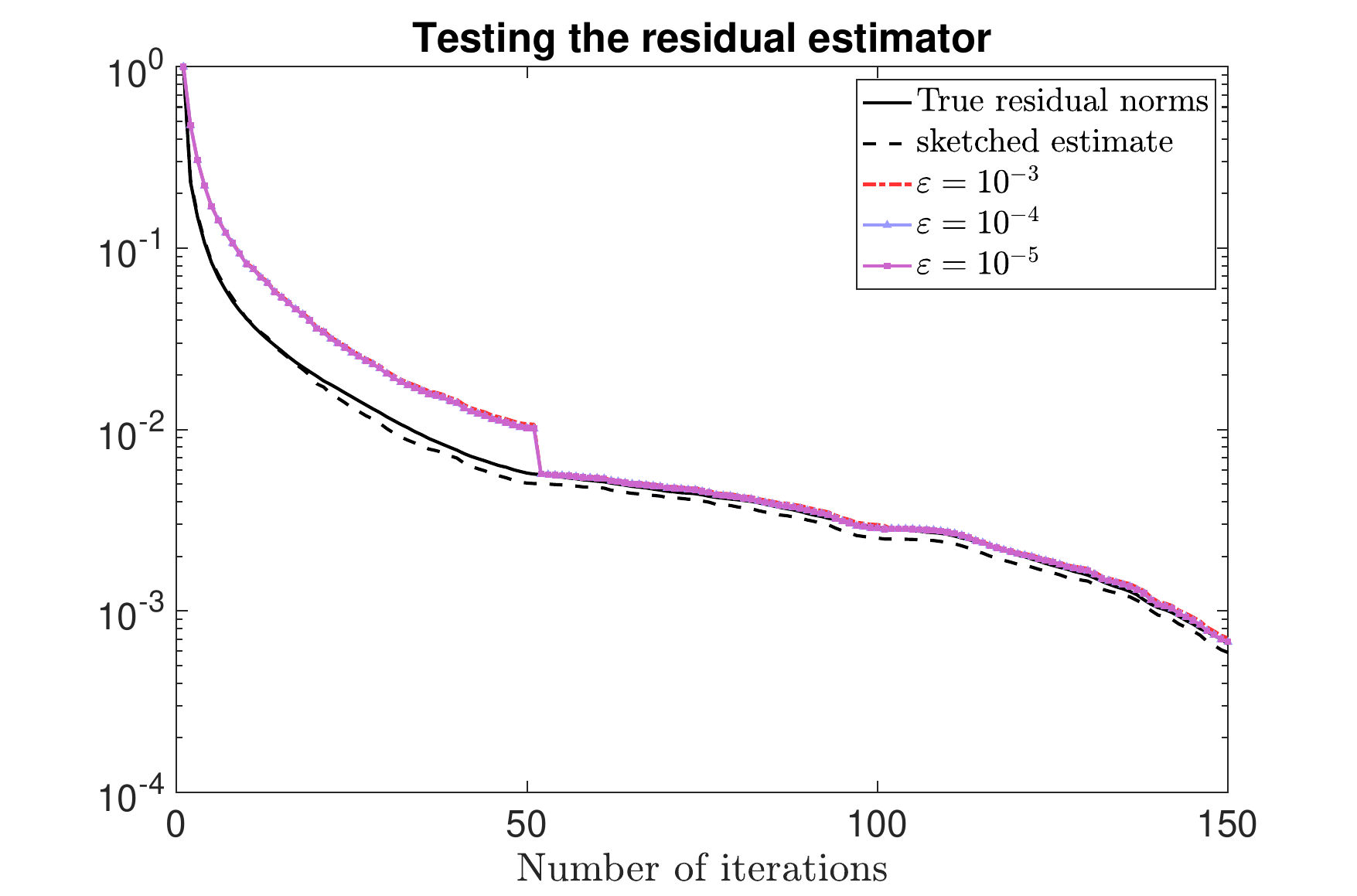}
    \caption{Testing the Givens rotation-based residual estimator $\norm{\tilde{\vr}_m}~\leq~\sqrt{\dfrac{s_m^2 + 2\varepsilon(1+c_m)} {1-\varepsilon^2}}\norm{\tilde{\vr}_{m-1}}$ for a problem of Neumann type from the experiments in \Cref{section.Neumann-experiments} for various proposed values of $\varepsilon$.}
    \label{fig:resid-estimator.pdf}
\end{figure}

\section{Numerical experiments}\label{sec:numerical-experiments}
We now present numerical experiments demonstrating the effectiveness of GMRES-SDR. We compare  to MATLAB's \texttt{gmres}, the RANDGMRES method in the \texttt{randKrylov}  package\footnote{\url{https://github.com/obalabanov/randKrylov}, version as of November 2023}, GCRO-DR and GMRES-DR. 
\rev{The sketching operator for GMRES-SDR is a subsampled randomized discrete cosine transform (DCT) in all cases: $S = PFE$, where $E\in\mathbb{R}^{N\times N}$ is a diagonal matrix with  diagonal entries $\pm 1$ with equal probability, $F\in\mathbb{R}^{N\times N}$ is a DCT, and $P\in\mathbb{R}^{s\times N}$ selects $s$ rows of $FE$ at random; see also \cite[Sec~8.1.1.]{NakatsukasaTropp21}. Such a sketching operator can be applied in $O(N\log s)$ operations; see, e.g.,~\cite[Section~3.3]{WoolfeLibertyRokhlinTygert2008}. The sketching parameter is fixed at $s = 8m_{\max}$, where $m_{\max}$ is the maximum Krylov dimension we encounter.} All experiments were performed in MATLAB Online, using version R2023B.

\subsection{Stokes problem}
In the first experiment we solve a single linear system $A \vx = \vb$, where $A$ is the \texttt{vas\_stokes\_1M} matrix of size $N=1,090,664$  from the SuiteSparse Matrix Collection~\cite{davis2011university}. We solve the system with ILU preconditioning to a target relative residual of $10^{-6}$. The residual curves are shown in \Cref{fig:stokes}~(left).
The true residual norms for RANDGMRES and GMRES-SDR are only computed and plotted at the end of each restart cycle as within each cycle only the sketched residual is cheaply available.

\begin{figure}
\hspace*{-0mm}\includegraphics[
 width=.5\textwidth]{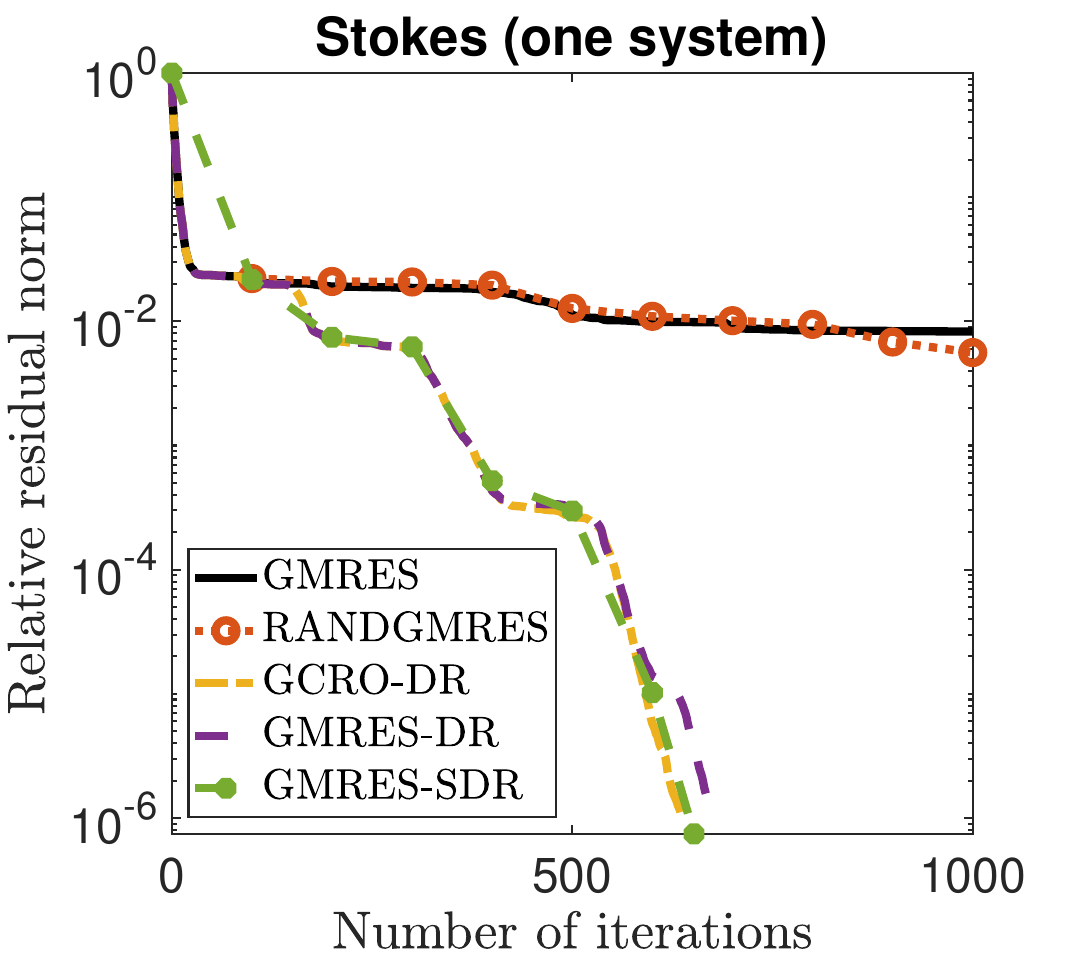}
\hspace*{-3mm}\includegraphics[
 width=.5\textwidth]{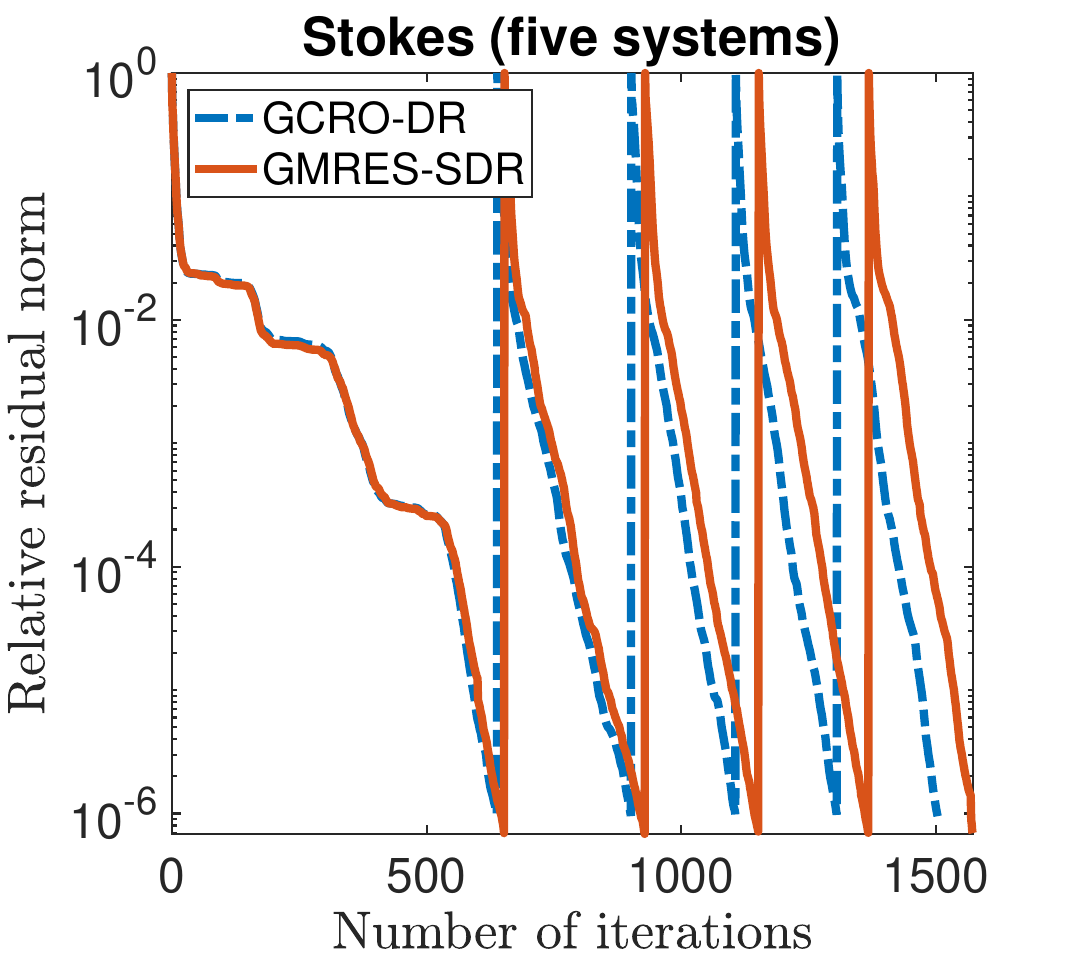}
\caption{Convergence curves obtained from solving a single linear system $A\vx = \vb$ (left), and a sequence of $5$ linear systems (right)  with the  \texttt{vas\_stokes\_1M} matrix, to a residual tolerance of $10^{-6}$. All non-augmented methods use a maximum number of $m = 100$ Arnoldi iterations, while the augmented methods use an augmentation subspace of dimension $k = 20$, and take a maximum of $m-k$ iterations. All sketching methods take $s = 10(m+k)$ and an Arnoldi truncation parameter $t = 2$. A maximum of $10$ restarts is allowed for all methods.}
\label{fig:stokes}
\end{figure}

It is clear from \Cref{fig:stokes} (left) that the convergence of GMRES-SDR is accelerated by  augmented restarting, and it produces convergence curves comparable with GCRO-DR and GMRES-DR.

In \Cref{tab:stokes_dr} we  record the total computational resources utilized by each method to converge, which we measure using the total number of matrix-vector products with~$A$ (denoted MATVECS or
MV), the total number of inner products (IP) of size~$N$, and the overall runtime (T) in seconds, required to produce the convergence curves in \Cref{fig:stokes} (left). 

\begin{table}
\centering
\begin{tabular}{p{0.8cm}| c c c c c}
   & \rev{GMRES} &  RANDGMRES & GCRO-DR  & GMRES-DR & GMRES-SDR\\
   \hline 
   MV & 1,000 & 1,000 & 670 &   674 & 656\\
    IP  & 55,550 & 11 & 51,138 & 45,229 & 1,955\\
    T (s) & 287.6 & 408.8  & 196.6 & 190.7 & 178.9   \\ 
\end{tabular}
 \caption{Performance metrics for solving a single linear system $A\vx = \vb$ with the  \texttt{vas\_stokes\_1M} matrix to a residual tolerance of $10^{-6}$ and with ILU preconditioning. \rev{GMRES\- and RANDGMRES do not reach the targeted residual tolerance within the maximal number of 10 restarts.}}
 \label{tab:stokes_dr}
\end{table} 

We also solve a sequence of $10$ systems $A\vx^{(i)}=\vb^{(i)}$ with the same fixed matrix and randomly generated right-hand sides (unit Gaussian). The performance metrics are given in \Cref{tab:stokes_10systems}. 

\begin{table}
\centering
\begin{tabular}{p{0.8cm}| c c c c c}
   & GMRES &  RANDGMRES & GCRO-DR  &  GMRES-DR  & GMRES-SDR \\
   \hline 
     MV & 10,000 & 10,000 & 2,646 & 6,943 & 3,439   \\
    IP  & 555,500 & 110 & 208,398 &  468,478  & 10,255 \\
    T (s) & 2535.2  & 4070.8 & 1474.7 &  1924.4 & 914.8  \\ 
\end{tabular}
 \caption{Performance metrics for solving a sequence of $10$ linear systems $A\vx^{(i)} = \vb^{(i)}$  with the  \texttt{vas\_stokes\_1M} matrix to a residual tolerance of $10^{-6}$ and with ILU preconditioning.  \rev{GMRES\- and RANDGMRES do not reach the targeted residual tolerance within the maximal number of 10 restarts for any of the ten problems.}}
\label{tab:stokes_10systems}
\end{table} 

From \Cref{tab:stokes_dr,tab:stokes_10systems} we see in that the combination of recycling and sketching in GMRES-SDR leads to a beneficial reduction in MATVECS and inner products, allowing for the fastest runtime for a single system  or a  sequence of systems. 

In \Cref{fig:stokes} (right) we plot the relative residual curves of both GCRO-DR and GMRES-SDR for the first $5$~systems in the sequence of Stokes problems. We see that the convergence of both methods is almost identical, and it is significantly improved with recycling beyond the first problem in the sequence.

\subsection{Neumann problem}\label{section.Neumann-experiments}
We now solve $50$ linear systems $A\vx^{(i)}=\vb^{(i)}$ for fixed matrix $A$ constructed as $A = D + c I$ where $D$ is the ``Neumann'' matrix of size $10,609$ taken from MATLAB's \texttt{gallery}, $I$ is the identity matrix and $c = 0.0001$. The right-hand side vectors $\vb^{(i)}$ were generated with random unit Gaussian entries. The utilized computational resources are shown in \Cref{tab:neumann_50systems}. The reported timings are averages over~$10$ runs of each solver.

\begin{table}
\centering
\begin{tabular}{p{0.8cm}| c c c c c}
   & GMRES &  RANDGMRES & GCRO-DR  &  GMRES-DR  &  GMRES-SDR \\
   \hline 
   MV & 50,000  & 50,000  & 8,978  & 20,528   &   6,906 \\
    IP & 2,777,500 & 550 & 531,179 &  1,164,290   &20,556\\
    T (s) & 122.1  & 74.7 & 17.3  &  34.2 & 11.8 \\ 
\end{tabular}
 \caption{Performance metrics for solving $50$ linear systems $A\vx^{(i)} = \vb^{(i)}$ with the Neumann matrix to a residual tolerance of $10^{-6}$. All non-augmented methods use a maximum number of $m = 100$ Arnoldi iterations, while the augmented methods use an augmentation subspace of dimension $k = 20$, and take a maximum of $m - k=80$ iterations. All sketching methods use  $s = 10(m+k)$ and an Arnoldi truncation parameter $t = 2$. A maximum of $10$ restarts is allowed for all methods. \rev{GMRES\- and RANDGMRES do not reach the targeted residual tolerance within the maximal number of 10 restarts for any of the fifty problems.}}
\label{tab:neumann_50systems}
\end{table} 

In \Cref{fig:neumann} we  plot the relative residual  curves obtained from solving both the first (left) and last (right) problem, respectively. We see that the GMRES-SDR convergence behaves in the same way as GCRO-DR, and substantially improves by the time the last problem is solved. We observe that GCRO-DR converges in slightly fewer iterations than GMRES-SDR, although as shown in \Cref{tab:neumann_50systems}, the overall runtime of GMRES-SDR is considerably lower. We also note that while GMRES-DR improves the convergence  through deflated restarts for each problem separately, GCRO-DR and GMRES-SDR lead to an overall faster convergence as they benefit from a recycling subspace that is continuously  improved as the problem sequence progresses. This is a key characteristic of a successful recycling method.

\begin{figure}

\centering
\hspace*{-0mm}\includegraphics[
 width=.50\textwidth]{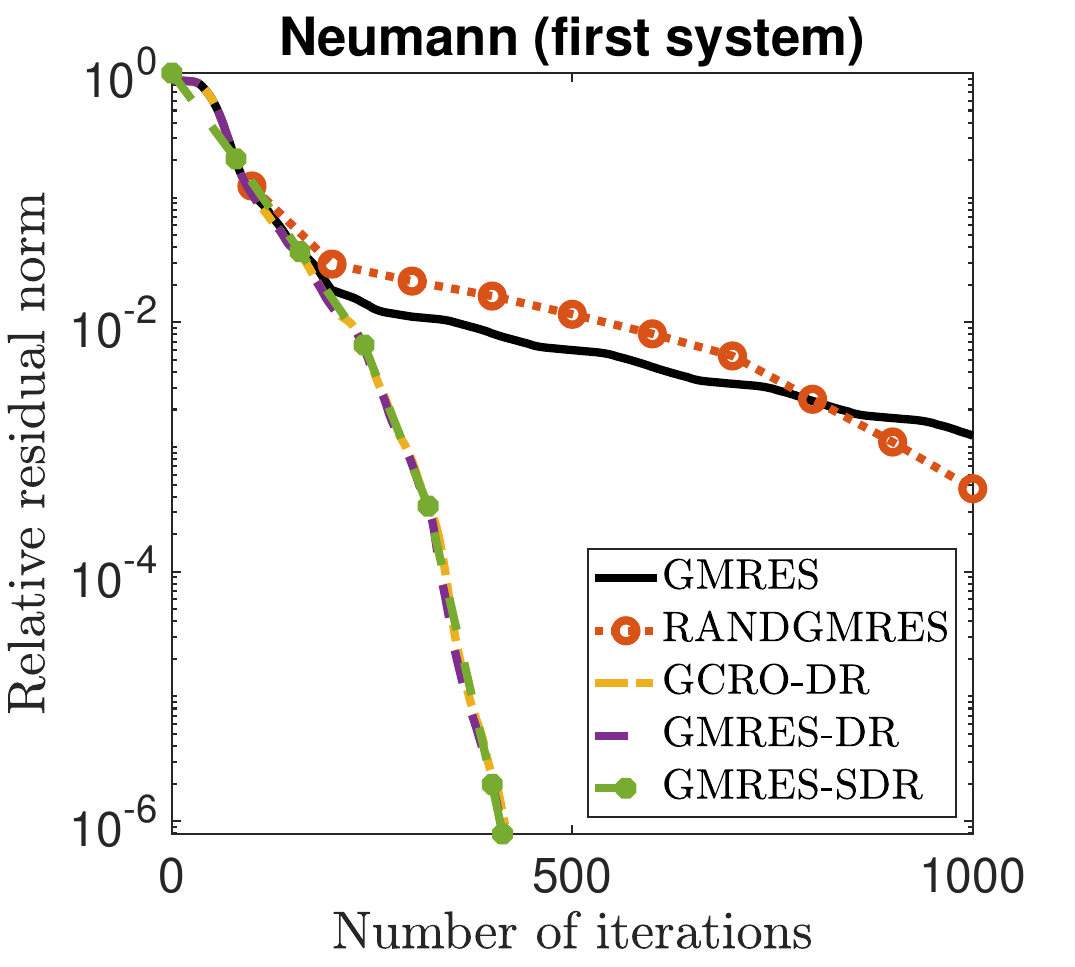}
\hspace*{-2mm}\includegraphics[ 
 width=.50\textwidth]{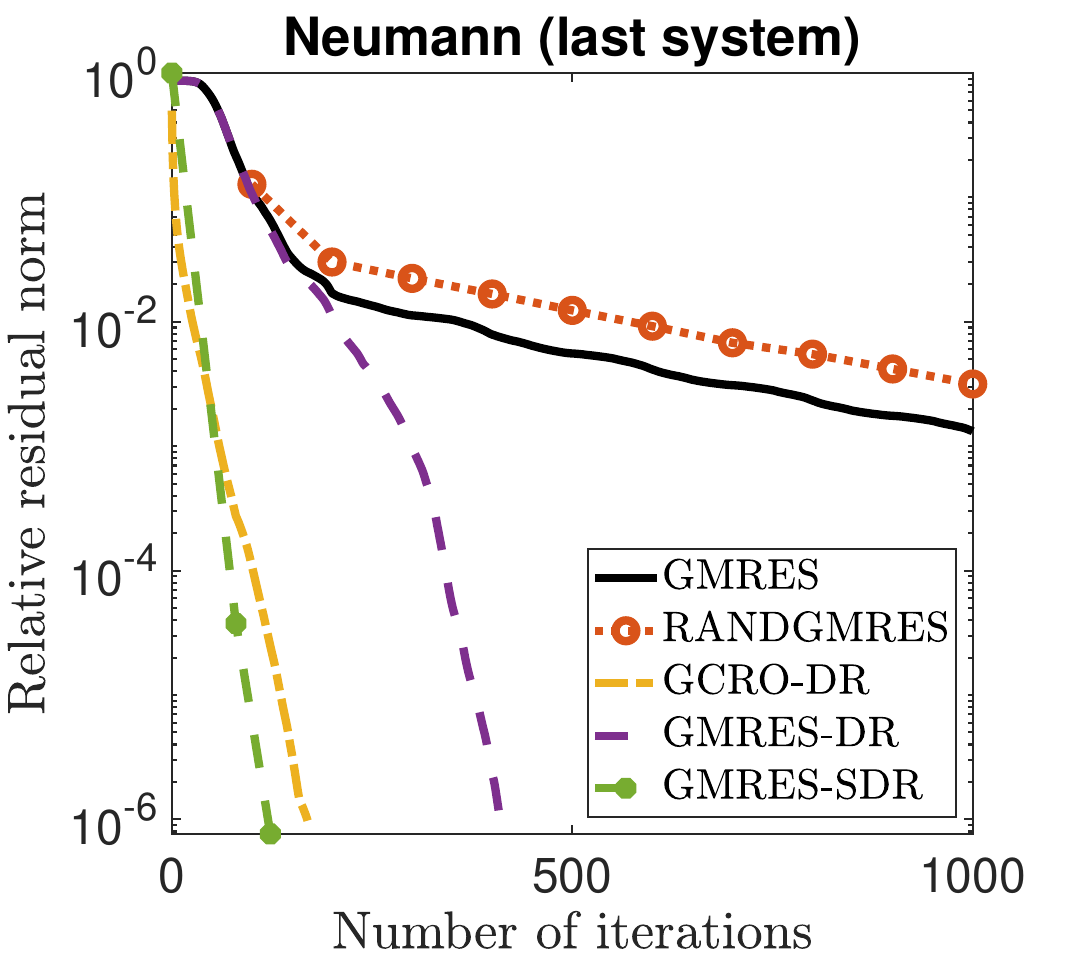}
\caption{Convergence curves for the first (left) and last (right) Neumann problem.}
\label{fig:neumann}
\end{figure}

\subsection{Convection--diffusion problem}\label{sec:cd} Finally, we consider a problem with varying system matrix $A^{(i)}$ arising from the finite-difference discretization of a 2D convection--diffusion problem for different convection strengths. More precisely, we use 
$A^{(i)} = (L \otimes I + I \otimes L) + \alpha^{(i)} ( D \otimes I + I \otimes D)$, where
\[
L = (n+1)^2\begin{bmatrix}
    -2 & 1 \\
     1 & -2 & \ddots \\
     & \ddots & \ddots & 1 \\
     & & 1 & -2
\end{bmatrix},\ 
D = \frac{n+1}{2}\begin{bmatrix}
    0 & 1 \\
     -1 & 0 & \ddots \\
     & \ddots & \ddots & 1 \\
     & & -1 & 0
\end{bmatrix} \in\mathbb{R}^{n\times n}, \  n = 500.
\]
The vector $\vb$ is chosen as the vector of all ones.  The convection strength is chosen as $\alpha^{(0)}=0,\alpha^{(1)}=5,\alpha^{(2)}=20$. The convergence of GCRO-DR and GMRES-SDR is plotted in \Cref{fig:cd} and the performance metrics are listed in \Cref{tab:cd}. In all cases, we a use Krylov dimension of $m=80$ and an augmentation dimension of $k=20$. For GMRES-SDR we consider the two variants discussed in \Cref{sec:slow}. The first one, labeled ``inexact'', corresponds to using $U, SU, SA^{(i)} U$ as the augmentation for problem~$i+1$. In other words, we are \emph{not} computing $S A^{(i+1)} U$ with the new matrix~$A^{(i+1)}$. 

There are several interesting observations to be made. First of all, GCRO-DR converges robustly for all three problems, and the convergence seems to indeed benefit from the augmentation. For the first two problems with $A^{(1)}$ and $A^{(2)}$, GMRES-SDR (inexact) and GMRES-SDR (exact) converge almost identically and are comparable to GCRO-DR.  However, for the third problem with $A^{(3)}$, GMRES-SDR (inexact) fails  with an erratically increasing residual. Note how this increase is not tracked by the residuals of the sketched problems  (shown as the yellow dotted curve), which are still decreasing within each restart cycle. As explained in \Cref{sec:slow}, the inexactness destroys the relation \eqref{eq:reschain} between the true and sketched residual problems. This is not the case for GMRES-SDR (exact), which decreases the sketched residual steadily (and the true residual almost monotonically) as expected from a  GMRES-type method. GMRES-SDR (inexact) fails because  the change in going from $A^{(2)}$ to $A^{(3)}$ is ``too large'', while it appears to be okay to go from $A^{(1)}$ to $A^{(2)}$ without computing $S A^{(2)} U$ explicitly.

\begin{figure}
\hspace*{-0mm}\includegraphics[width=.349\textwidth]{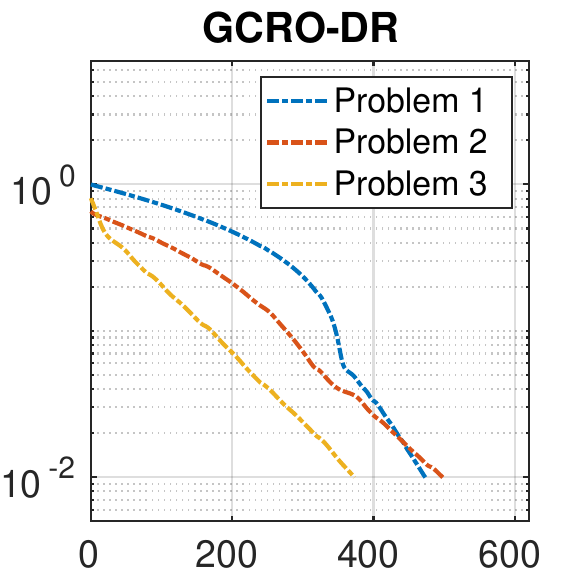}
\hspace*{-4mm}\includegraphics[width=.349\textwidth]{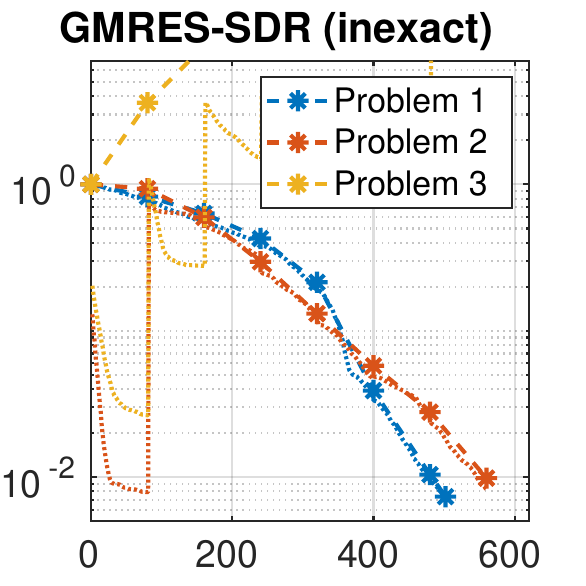}
\hspace*{-4mm}\includegraphics[width=.349\textwidth]{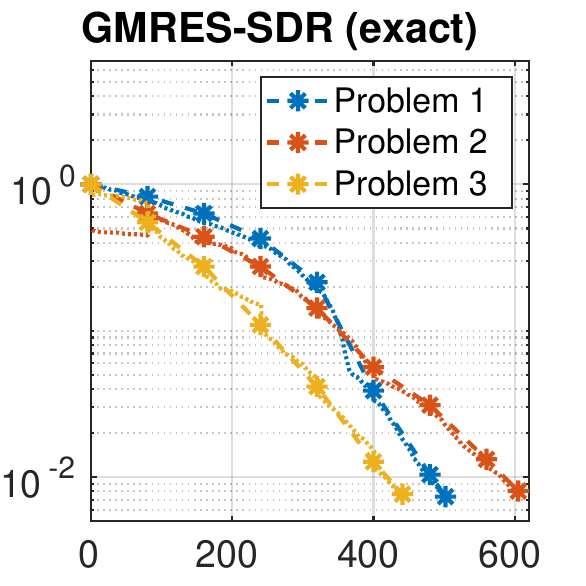}
\caption{Convergence curves for three convection--diffusion problems with increasing convection strength. For the two GMRES-SDR variants, there are two curves for each problem: (i) a dashed curve  with ``*'' markers showing the true residual after each restart cycle, and (ii) a dotted curve corresponding to the residual of the sketched problem.}
\label{fig:cd}
\end{figure}

\begin{table}
\centering
\begin{tabular}{p{2.3cm}|c c c}
   & GCRO-DR &  GMRES-SDR  & GMRES-SDR \\
   & & (inexact) & (exact)\\
   \hline 
   MATVECS & 1,345 & 1,886 & 1,971 \\
   Inner products & 105,695 & 5,613 & 4,668\\
   Time (s) & 38.7 & 30.8 &  27.0 \\ 
\end{tabular}
 \caption{Solving three convection--diffusion problems $A^{(i)}\vx^{(i)} = \vb$. GMRES-SDR (inexact) fails to converge to the target   residual norm of $10^{-2}$ for the third problem.}
 \label{tab:cd}
\end{table}

\section{Conclusions and future work}  We have introduced GMRES-SDR, a GMRES variant that combines sketching with deflated restarting. Our numerical tests indicate that GMRES-SDR can improve over GCRO-DR and GMRES-DR in terms of arithmetic cost and runtime. The runtime reduction compared to GCRO-DR (the best competing method) ranged between 30\% and 38\% on the problems we considered. This is primarily the result of a reduction in inner products due to sketching, and sometimes even a reduction in matrix-vector products due to improved deflation (see, e.g., Table~\ref{tab:neumann_50systems} where GMRES-SDR required about 23\% fewer matrix-vector products than GCRO-DR).

Our analysis extended and generalized several results known for GMRES to the sketched and augmented case. In particular, we characterized  GMRES-SDR as a projection method using a semi-inner product. 

In future work we would like to understand better the behavior of the ``inexact'' GMRES-SDR variant. In particular, it would be useful to know a priori what constitutes a sufficiently slow variation in $A^{(i)}$ so that the inexact variant still converges. 

\rev{
\section*{Acknowledgements}
The authors would like to thank the referees and the handling editor for their constructive critiques and comments.  Implementing them greatly improved the quality of the exposition in this paper.
}

\color{black}

\bibliographystyle{plain}
\bibliography{refs}

\end{document}